\documentclass{amsart}
\usepackage{amsmath}
\usepackage{amssymb}
\usepackage{enumerate}
\usepackage{mathrsfs}
\usepackage{cases}
\usepackage{color}
\usepackage{tikz}
\usepackage{appendix}
\usepackage{makecell}
\usepackage{multirow}
\usepackage{algpseudocode}
\usepackage{algorithmicx,algorithm}
\usepackage{graphicx}
\definecolor{darkblue}{rgb}{0.0, 0.0, 0.55}
\definecolor{bordeaux}{rgb}{0.34, 0.01, 0.1}
\usepackage[pagebackref,colorlinks,linkcolor=bordeaux,citecolor=darkblue,urlcolor=black,hypertexnames=true]{hyperref}
\usetikzlibrary{arrows,positioning}
\usepackage{siunitx}

\newtheorem{theorem}{Theorem}[section]

\newtheorem{example}[theorem]{Example}

\newtheorem{remark}[theorem]{Remark}

\def\R{{\mathbb{R}}}

\def\N{{\mathbb{N}}}
\def\BB{{\mathbb{B}}}

\def\x{{\mathbf{x}}}

\def\y{{\mathbf{y}}}

\def\bu{{\mathbf{u}}}

\def\W{{\mathbf{W}}}

\def\b{{\boldsymbol{\beta}}}
\def\g{{\boldsymbol{\gamma}}}
\def\bv{{\boldsymbol{v}}}

\def\A{{\mathbf{A}}}
\def\B{{\mathbf{B}}}
\def\S{{\mathbb{S}}}
\def\M{{\mathcal{M}}}
\def\D{{\mathcal{D}}}

\def\CC{{\mathscr{C}}}
\def\DD{{\mathscr{D}}}

\def\supp{\hbox{\rm{supp}}}
\def\var{\hbox{\rm{var}}}

\def\tr{\hbox{\rm{tr}}}
\def\Tr{\hbox{\rm{Tr}}}

\def\int{\hbox{\rm{int}}}

\def\SE{\hbox{\rm{SE}}}
\def\CSE{\hbox{\rm{CSE}}}

\def\Sym{\hbox{\rm{Sym}}}
\def\cdeg{\hbox{\rm{cdeg}}}
\DeclareMathOperator{\II}{II}
\newcommand{\vge}{\mathbin{\rotatebox[origin=c]{90}{$\ge$}}}
\newif\ifcomment
\commentfalse
\commenttrue
\usepackage{todonotes}

\begin{document}

\title{Exploiting term sparsity in Noncommutative Polynomial Optimization}
\author{Jie Wang \and Victor Magron}
\subjclass[2010]{Primary, 47N10,90C22; Secondary, 12D15,14P10}
\keywords{term sparsity, noncommutative polynomial optimization, semidefinite relaxations, sum of hermitian squares}
\date{\today}

\begin{abstract}
We provide a new hierarchy of semidefinite programming relaxations, called \emph{NCTSSOS}, to solve large-scale sparse noncommutative polynomial optimization problems. 
This hierarchy features the exploitation of \emph{term sparsity} hidden in the input data for eigenvalue and trace optimization problems. 
NCTSSOS complements the recent work that exploits \emph{correlative sparsity} for noncommutative optimization problems by Klep, Magron and Povh \cite{klep2019sparse}, and is the noncommutative analogue of the TSSOS framework by Wang, Magron and Lasserre \cite{wang2,wang3}.
We also propose an extension exploiting simultaneously correlative and term sparsity, as done previously in the commutative case \cite{wang4}.
Under certain conditions, we prove that the optimums of the NCTSSOS hierarchy converge to the optimum of the corresponding dense SDP relaxation.
We illustrate the efficiency and scalability of NCTSSOS by solving eigenvalue/trace optimization problems from the literature as well as randomly generated examples involving up to several thousands of variables.
\end{abstract}

\maketitle

\section{Introduction}

A polynomial optimization problem (POP) consists of minimizing a polynomial over a basic closed \emph{semialgebraic set}, namely an intersection of finitely many polynomial level sets.
Even if solving a POP is NP-hard in general \cite{Laurent:Survey}, one can rely on the so-called ``moment-sums of squares (SOS) hierarchy'', also referred to as ``Lasserre's hierarchy'' \cite{Las01} to compute a sequence of lower bounds for the POP. 
Each lower bound in the sequence is obtained by solving a semidefinite program (SDP) \cite{anjos2011handbook}.
Thanks to Putinar's Positivstellensatz \cite{Putinar1993positive}, if the quadratic module generated by the polynomials describing the semialgebraic set is \emph{Archimedean}, the sequence of these SDP lower bounds converges from below to the global optimal value of the POP.

Although most POPs involve commuting variables, we are also interested in noncommutative POPs (NCPOPs), i.e., POPs involving noncommuting variables (e.g. matrices, operators on a Hilbert space). The applications of NCPOPs include control theory and linear systems in engineering \cite{skelton1997}, quantum theory and quantum information science \cite{gribling2018,pal2009,marecek2020}, matrix factorization ranks \cite{gribling2019}, machine learning \cite{zhou2020proper,zhou2020fairness} and so on, and new applications are emerging.

The noncommutative (nc) analogue of Lasserre's hierarchy \cite{cafuta2012constrained,Helton04,navascues2008convergent,pironio2010convergent}, often called the ``Navascu\'es-Pironio-Ac\'in (NPA) hierarchy'', or ``moment-sums of hermitian squares (SOHS) hierarchy'', allows one to compute arbitrarily close lower bounds of the minimal eigenvalue of an nc polynomial over an nc semialgebraic set.
In the same spirit, one can also obtain a hierarchy of SDP relaxations to approximate as closely as desired the minimal trace of an nc polynomial over an nc semialgebraic set \cite{nctrace,cafuta2012constrained,pironio2010convergent}.
We also refer the interested reader to \cite{klep2020} for the case of more general trace polynomials, i.e., polynomials in noncommutating variables and traces of their products.

From the view of applications, the common bottleneck of the Lasserre/NPA hierarchy is that the involved sequence of SDP relaxations becomes intractable very quickly as the number of variables $n$ or the relaxation order $d$ increases. In the commutative setting, the matrices involved at relaxation order $d$ is of size $\binom{n+d}{d}$; in the nc setting, the size of matrices involved at relaxation order $d$ is even larger. It is already hard to solve such a SDP for $n\le30$ and $d\ge2$ on a modern standard laptop (at least when one relies on interior-point solvers).

~

\paragraph{\textbf{Remedies by exploiting sparsity for (NC)POP}}

In certain situations, the SDP relaxations arising from POPs can be solved with adequate first-order methods rather than with costly interior-point algorithms; see, e.g., \cite{mai2020hierarchy,yurtsever2019scalable}, where the authors exploit the constant trace property of the matrices in SDP relaxations of combinatorial optimization problems or quadratically constrained quadratic programs.
In any case, it is worth finding remedies in view of the sparsity of (NC)POPs to prevent from the computational blow-up of the Lasserre/NPA hierarchy, by decreasing the sizes of the matrices involved in the SDP relaxations. 

The first remedy is to partition the input variables into cliques when the polynomials involved in the objective function and the constraints fulfill a so-called \textit{correlative sparsity pattern}.
The resulting moment-SOS hierarchy is obtained by assembling sparse SDP matrices in terms of these cliques of variables \cite{waki,waki2008algorithm}. Under certain conditions, the lower bounds given by this hierarchy still converge to the global optimum of the original problem \cite{Las06}. When the sizes of these cliques of variables are small enough (e.g., less than $10$ in \cite{waki} or less than $20$ in \cite{josz2018lasserre}), one can significantly improve the scalability of Lasserre's hierarchy to handle problems with a large number of variables. For instance, by exploiting correlative sparsity, one can compute roundoff error bounds \cite{toms17,toms18} with up to hundred variables, and solve optimal power flow problems \cite{josz2018lasserre} or deep learning problems \cite{chen2020polynomial} with up to thousands of variables. 
Several extensions have been investigated, including volume computation of sparse semialgebraic sets \cite{tacchi2019exploiting}, or minimization of rational functions \cite{bugarin2016minimizing,mai2020sparse}. 
Recently, Klep, the second author and Povh designed an nc analogue of this sparsity adapted hierarchy for both eigenvalue and trace minimization problems \cite{klep2019sparse}. Nevertheless, when the  sizes of variable cliques provided by correlative sparsity are relatively big (say $\ge20$ in the commutative setting and $\ge10$ in the nc setting), or when the correlative sparsity pattern is even fully dense, one might face again the same issue of untractable SDPs.


Another complementary remedy consists of exploiting \emph{term sparsity}. 
For unconstrained problems, it means that the objective function involves few terms (monomials or words).
One can then reduce the size of the associated SDP matrix by computing a smaller monomial basis via the \emph{Newton polytope method} \cite{re2}.
The nc analogue for this method is the \emph{Newton chip method}~\cite[\textsection2.3]{burgdorf16} in the context of eigenvalue optimization and the \emph{tracial Newton polytope}~\cite[\textsection3.3]{burgdorf16} in the context of trace optimization.
%

Besides obtaining a smaller monomial basis, in both unconstrained and constrained case, one can rely on a term-sparsity adapted moment-SOS hierarchy (called TSSOS), following the line of research recently pursued by the two authors, Lasserre and Mai in \cite{wang,wang2,wang3}.
The core idea of TSSOS is partitioning the monomial bases used to construct SDP relaxations into blocks in view of the correlations between monomials and then building SDP matrices to comply with this block structure. More precisely, one first define the {\em term sparsity pattern (tsp) graph} associated with a POP whose nodes are monomials from the monomial basis. Two nodes of the tsp graph are connected via an edge if and only if the product of the corresponding monomials belongs to the support of the polynomials involved in the POP or is a monomial square.
TSSOS is based on an iterative procedure, whose input is the tsp graph of the POP. Each iteration consists of two steps: first one performs a support-extension operation on the graph and successively one performs a chordal-extension operation on the graph (``maximal'' chordal extensions are used in \cite{wang2} while approximately minimal chordal extensions are used in \cite{wang3}). At each iteration, one can construct a SDP relaxation with matrices of sparsity pattern represented by the corresponding graph.
In doing so, TSSOS provides us with a two-level moment-SOS hierarchy involving sparse SDP matrices. TSSOS can be further combined with correlative sparsity, which allows one to solve large-scale POPs with several thousands of variables and constraints \cite{wang4}. Apart from (commutative) POPs, the idea of TSSOS can be also used to develop more efficient SOS-based algorithms for other problems, e.g. the approximation of joint spectral radius \cite{wang2020sparsejsr}. 
%

~

\paragraph{\textbf{Contributions}}
Motivated by the performance of TSSOS for commutative POPs, we develop an nc analogue of the TSSOS framework (called NCTSSOS) in this paper, to handle large-scale eigenvalue/trace optimization problems with sparse input data.

First, we extend in Section~\ref{sec3} the notion of term sparsity pattern to unconstrained eigenvalue optimization problems. We show how to build the tsp graph and derive a two-step iterative procedure to enlarge the graph via the support-extension operation and the chordal-extension operation. Based on this, we then give the NCTSSOS hierarchy.
The generalization to constrained eigenvalue optimization problems is provided in Section~\ref{sec3-con}. Under certain conditions, we prove that the optimums of the NCTSSOS hierarchy converge to the optimum of the dense relaxation.
In Section~\ref{eo-cts}, we show how to benefit simultaneously from both correlative and term sparsity, to obtain an nc variant of the so-called ``CS-TSSOS'' hierarchy \cite{wang4}. 
Section~\ref{sec5} is dedicated to trace optimization.
For both unconstrained and constrained trace optimization problems, we provide a term sparsity (and combined with correlative sparsity) adapted hierarchy of SDP relaxations.
In Section~\ref{sec:benchs}, we demonstrate the computational efficiency, scalability and accuracy of the NCTSSOS hierarchy by various numerical examples involving up to several thousands of variables.

The algorithmic framework of the NCTSSOS hierarchy has been released as an open-source Julia \cite{bezanson2017julia} library, also called \texttt{NCTSSOS},
which is available online and comes together with a documentation.\footnote{\url{https://github.com/wangjie212/NCTSSOS}}

Our term sparsity (and combined with correlative sparsity) adapted moment-SOHS hierarchies appear in a similar manner as the ones obtained for the commutative case \cite{wang2,wang3,wang4}. 
We believe that it is of interest for researchers using noncommutative optimization tools to have a self-contained paper stating explicitly the construction of the tsp graphs, support/chordal-extension operations, as well as the different term sparsity (and combined with correlative sparsity) adapted SDP formulations, either for eigenvalue or trace optimization.
While the overall strategy to obtain tsp graphs for eigenvalue optimization is very similar to the commutative case, it is less straightforward for trace optimization, where it is mandatory to introduce the cyclic analog of the tsp graph and the support-extension operation. 
Furthermore, we would like to emphasize that the main contribution of this paper is to show a significant quantitative improvement with respect to the previous results obtained for various nc eigenvalue/trace optimization problems. We hope that these results will convince researchers in related fields, including quantum information physicists relying on the NPA hierarchy, about the potential impact that NCTSSOS could have on solving their problems more efficiently.

\section{Notation and Preliminaries}
In this section, we recall some notations, definitions and basic results that will be used in the rest of this paper.
\subsection{Noncommutative polynomials}
For a positive integer $r$, let us denote by $\S^r$ (resp. $\S_+^r$, $\S_{++}^r$) the
space of all symmetric (resp. positive semidefinite (PSD), positive definite) matrices of size $r$, and by $(\S^r)^n$
the set of $n$-tuples $\underline{A} = (A_1, \ldots, A_n)$ of symmetric matrices $A_i$ of size $r$. For matrices $A,B\in\S^r$ (resp. vectors $\bu,\bv\in\R^r$), let $\langle A, B\rangle\in\R$ (resp. $\langle\bu, \bv\rangle\in\R$) be the trace inner-product, defined by $\langle A, B\rangle=\Tr(A^TB)$ (resp. $\langle\bu,\bv\rangle=\bu^T\bv$) and let $A\circ B\in\S^r$ denote the Hadamard, or entrywise, product of $A$ and $B$, defined by $[A\circ B]_{ij} = A_{ij}B_{ij}$.
For a fixed $n\in\N\backslash\{0\}$, let $\underline{X}=(X_1, \ldots, X_n)$ be a tuple of letters and consider the set of all possible words of finite length in $\underline{X}$ which is denoted by $\langle\underline{X}\rangle$. The empty word is
denoted by $1$. We denote by $\R\langle\underline{X}\rangle$ the ring of real polynomials in the noncommutating variables $\underline{X}$. An element $f$ in $\R\langle\underline{X}\rangle$ can be written as $f=\sum_{w\in\langle\underline{X}\rangle}a_ww$, $a_w\in\R$, which is called a {\em noncommutative polynomial} ({\em nc polynomial} for short). The {\em support} of $f$ is defined by $\supp(f):=\{w\in\langle\underline{X}\rangle\mid a_w\ne0\}$ and the {\em degree} of $f$, denoted by $\deg(f)$, is the length of the longest word in $\supp(f)$. For a given $d\in\N$, let us denote by $\W_d$ the column vector of all words of degree at most $d$ arranged w.r.t. the lexicographic order. The ring $\R\langle\underline{X}\rangle$ is equipped with the involution $\star$ that fixes $\R\cup\{X_1, \ldots, X_n\}$ point-wise and reverses words, so that $\R\langle\underline{X}\rangle$ is the $\star$-algebra freely generated by $n$ symmetric letters $X_1, \ldots, X_n$. The set of symmetric elements in $\R\langle\underline{X}\rangle$ is defined as $\Sym\,\R\langle\underline{X}\rangle:=\{f\in\R\langle\underline{X}\rangle\mid f^{\star}=f\}$. We use $|\cdot|$ to denote the cardinal of a set and let $[m]:=\{1,2,\ldots,m\}$ for $m\in\N\backslash\{0\}$. 

\subsection{Sums of hermitian squares}
An nc polynomial of the form $g^{\star}g$ is called a {\em hermitian square}. A nc polynomial $f\in\R\langle\underline{X}\rangle$ is called a {\em sum of hermitian squares (SOHS)} if there exist nc polynomials $g_1, \ldots, g_r\in\R\langle\underline{X}\rangle$ such that $f = g_1^{\star}g_1+g_2^{\star}g_2+\ldots+g_r^{\star}g_r$. The set of SOHS is denoted by $\Sigma\langle\underline{X}\rangle$. Checking whether a given
nc polynomial $f\in\Sym\,\R\langle\underline{X}\rangle$ is an SOHS can be cast as a semidefinite program (SDP) due to the following theorem.
\begin{theorem}[\cite{helton}, Theorem 1.1]
Let $f\in\Sym\,\R\langle\underline{X}\rangle$ with $\deg(f)=2d$. Then $f\in\Sigma\langle\underline{X}\rangle$ if and only if there exists a matrix $Q\succeq0$ satisfying
\begin{equation}\label{sec2-eq1}
f=\W_d^{\star}Q\W_d.
\end{equation}
\end{theorem}
Any symmetric matrix $Q$ (not necessarily PSD) satisfying \eqref{sec2-eq1} is called a {\em Gram matrix} of $f$. The standard monomial basis $\W_d$ used in \eqref{sec2-eq1} can be reduced via the Newton chip method; see Chapter 2 in \cite{burgdorf}.

\subsection{Semialgebraic sets and quadratic modules}
Given $S=\{g_1,\ldots,g_m\}\subseteq\Sym\,\R\langle\underline{X}\rangle$, the {\em semialgebraic set} $\D_S$ associated with $S$ is defined by
\begin{equation}
    \D_S:=\bigcup_{r\in\N\backslash\{0\}}\{\underline{A}=(A_1,\ldots,A_n)\in(\S^r)^n\mid g_j(\underline{A})\succeq0, j\in[m]\}.
\end{equation}
The {\em operator semialgebraic set} $\D_S^{\infty}$ is the set of all bounded self-adjoint operators $\underline{A}$ on a Hilbert space endowed with a scalar product $\langle\cdot,\cdot\rangle$ making $g(\underline{A})$ a PSD operator, for all $g\in S$. The {\em quadratic module} $\M_S$, generated by $S$, is defined by
\begin{equation}
    \M_S:=\{\sum_{j=1}^sh_j^{\star}g_jh_j\mid s\in\N\backslash\{0\}, h_j\in\R\langle\underline{X}\rangle, g_j\in \{1\}\cup S\},
\end{equation}
and the {\em truncated quadratic module} $\M_{S,d}$ of order $d\in\N$, generated by $S$, is
\begin{equation}
    \M_{S,d}:=\{\sum_{j=1}^sh_j^{\star}g_jh_j\mid s\in\N\backslash\{0\}, h_j\in\R\langle\underline{X}\rangle, g_j\in \{1\}\cup S,\deg(h_j^{\star}g_jh_j)\le2d\}.
\end{equation}
A quadratic module $\M$ is {\em Archimedean} if for each $h\in\R\langle\underline{X}\rangle$, there exists $N\in\N$ such that $N-h^{\star}h\in\M$. The noncommutative analog of Putinar’s Positivstellensatz describing noncommutative polynomials positive on $\D_S^{\infty}$ with Archimedean $\M_S$ is due to Helton and McCullough:
\begin{theorem}[\cite{helton2}, Theorem 1.2]
Let $\{f\}\cup S\subseteq\Sym\,\R\langle\underline{X}\rangle$ and assume that $\M_S$ is Archimedean. If $f(A)\succ0$ for all $A\in\D_S^{\infty}$, then $f\in\M_S$.
\end{theorem}

\subsection{Moment and localizing matrices}
With $\y=(y_{w})_{w\in\langle\underline{X}\rangle}$ being a sequence indexed by the standard monomial basis $\langle\underline{X}\rangle$ of $\R\langle\underline{X}\rangle$, let $L_{\y}:\R\langle\underline{X}\rangle\rightarrow\R$ be the linear functional
\begin{equation*}
f=\sum_{w}a_{w}w\mapsto L_{\y}(f)=\sum_{w}a_{w}y_{w}.
\end{equation*}
Given a monomial basis $\B$, the noncommutative {\em moment} matrix $M_{\B}(\y)$ associated with $\B$ and $\y$ is the matrix with rows and columns indexed by $\B$ such that
\begin{equation*}
M_{\B}(\y)_{uv}:=L_{\y}(u^{\star}v)=y_{u^{\star}v}, \quad\forall u,v\in\B.
\end{equation*}
If $\B$ is the standard monomial basis $\W_d$, we also denote $M_{\W_d}(\y)$ by $M_{d}(\y)$.

Suppose $g=\sum_{w}b_{w}w\in\Sym\,\R\langle\underline{X}\rangle$ and let $\y=(y_{w})_{w\in\langle\underline{X}\rangle}$ be given. For any positive integer $d$, the noncommutative {\em localizing} matrix $M_{d}(g\y)$ associated with $g$ and $\y$ is the matrix with rows and columns indexed by $\W_d$ such that
\begin{equation*}
M_{d}(g\y)_{uv}:=L_{\y}(u^{\star}gv)=\sum_{w\in\supp(g)}b_{w}y_{u^{\star}wv}, \quad\forall u,v\in\W_d.
\end{equation*}

\subsection{Eigenvalue optimization for noncommutative polynomials}\label{sec2-eo}
Given $f=\sum_{w}a_ww\in\Sym\,\R\langle\underline{X}\rangle$, the eigenvalue minimization problem for $f$ is defined by:
\begin{equation}\label{upop-eigen}
(\textrm{EP}_0):\quad\lambda_{\min}(f):=\inf\{\langle f(\underline{A})\bv,\bv\rangle:\underline{A}\in(\S^r)^n, r\in\N\backslash\{0\},||\bv||=1\}.
\end{equation}
Assume that $\B$ is a monomial basis. Then $(\textrm{EP}_0)$ is equivalent to the following SDP (\cite{burgdorf})
\begin{equation}\label{upop-eigen1}
(\textrm{EP}):\quad
\begin{array}{rll}
\lambda_{\min}(f)=&\inf&L_{\y}(f)\\
&\textrm{s.t.}&M_{\B}(\y)\succeq0,\\
&&y_{1}=1.
\end{array}
\end{equation}
Writing $M_{\B}(\y)=\sum_{w}A_{w}y_{w}$ for appropriate symmetric matrices $\{A_{w}\}_w$, the dual SDP of \eqref{upop-eigen1} is
\begin{equation}\label{upop-eigen2}
(\textrm{EP})^*:\quad
\begin{array}{ll}
\sup&\lambda\\
\textrm{s.t.} &\langle Q,A_{w}\rangle+\lambda\delta_{1w}=a_{w},\quad\forall w\in\B^{\star}\B,\\
&Q\succeq0,
\end{array}
\end{equation}
where $\B^{\star}\B:=\{u^{\star}v\mid u,v\in\B\}$ and $\delta_{1w}$ is the usual Kronecker symbol.

Given $f=\sum_{w}a_ww\in\Sym\,\R\langle\underline{X}\rangle$ and $S=\{g_1,\ldots,g_m\}\subseteq\Sym\,\R\langle\underline{X}\rangle$, let us consider the following eigenvalue minimization problem for $f$ over the operator semialgebraic set $\D_S^{\infty}$:
\begin{equation}\label{cpop-eigen}
(\textrm{EQ}_0):\quad\lambda_{\min}(f, S):=\inf\{\langle f(\underline{A})\bv,\bv\rangle:\underline{A}\in\D_S^{\infty}, ||\bv||=1\}.
\end{equation}
For convenience, we set $g_0:=1$ and let $d_j=\lceil\deg(g_j)/2\rceil$ for $j=0,1,\ldots,m$. Assume that $\hat{d}\ge d := \max\{\lceil\deg(f)/2\rceil,d_1,\ldots,d_m\}$ is a positive integer. As shown in \cite{pironio}, one has the following hierarchy of moment relaxations, indexed by $\hat{d}$, to obtain a sequence of lower bounds for the optimum $\lambda_{\min}(f, S)$ of ($\textrm{EQ}_0$):
\begin{equation}\label{cpop-eigen1}
(\textrm{EQ}_{\hat{d}}):\quad
\begin{array}{rll}
\lambda_{\hat{d}}(f, S):=&\inf &L_{\y}(f)\\
&\textrm{s.t.}&M_{\hat{d}}(\y)\succeq0,\\
&&M_{\hat{d}-d_j}(g_j\y)\succeq0,\quad j\in[m],\\
&&y_{1}=1.
\end{array}
\end{equation}
We call $\hat{d}$ the {\em relaxation order}. 
If the quadratic module $\M_S$ generated by $S$ is Archimedean then the sequence of lower bounds $(\lambda_{\hat{d}}(f, S))_{\hat d\geq d}$ converges to $\lambda_{\min}(f, S)$. 
See, e.g.,~\cite[Corollary~4.11]{burgdorf16} for a proof.

For each $j$, writing $M_{\hat{d}-d_j}(g_j\y)=\sum_{w}D_{w}^jy_{w}$ for appropriate symmetric matrices $\{D_{w}^j\}_{j,w}$, we can write the dual SDP of \eqref{cpop-eigen1} as:
\begin{equation}\label{cpop-eigen2}
(\textrm{EQ}_{\hat{d}})^*:\quad
\begin{array}{rll}
&\sup&\lambda\\
&\textrm{s.t.}&\displaystyle\sum_{j=0}^m\langle Q_j,D_{w}^j\rangle+\lambda\delta_{1w}=a_{w},\quad\forall w\in\W_{2\hat{d}},\\
&&Q_j\succeq0,\quad j\in\{0\}\cup[m].
\end{array}
\end{equation}

\subsection{Trace optimization for noncommutative polynomials}\label{sec2-to}
Given $g,h\in\R\langle\underline{X}\rangle$, the nc polynomial $[g, h]:=gh-hg$ is called a {\em commutator}. Two nc polynomials $g,h\in\R\langle\underline{X}\rangle$ are said to be {\em cyclically
equivalent}, denoted by $g\stackrel{\textrm{cyc}}{\sim} h$, if $g-h$ is a sum of commutators. Let $w\in\langle\underline{X}\rangle$. 
The canonical representative $[w]$ of $w$ is the minimal one with
respect to the lexicographic order among all words cyclically equivalent to $w$. For $\A\subseteq\langle\underline{X}\rangle$, $[\A]:=\{[w]\mid w\in\A\}$. For an nc polynomial $f=\sum_{w}a_ww\in\Sym\,\R\langle\underline{X}\rangle$, the canonical representative of $f$ is defined by $[f]:=\sum_{w}a_w[w]\in\R\langle\underline{X}\rangle$ and the {\em cyclic degree} of $f$ is defined as $\cdeg(f):=\deg([f])$.
We warn the reader about a small abuse of notation as $[k]$ stands for $\{1,\dots,k\}$ when $k$ is a positive integer.

The {\em normalized trace} of a matrix $A=[a_{ii}]\in\S^r$ is given by $\tr\,A=\frac{1}{r}\sum_{i=1}^ra_{ii}$. Given $f=\sum_{w}a_ww\in\Sym\,\R\langle\underline{X}\rangle$, the trace minimization problem for $f$ is defined by:
\begin{equation}\label{upop-trace}
(\textrm{TP}_0):\quad\tr_{\min}(f):=\inf\{\tr\,f(\underline{A}):\underline{A}\in(\S^r)^n, r\in\N\backslash\{0\}\}.
\end{equation}
Let $d=\cdeg(f)$. As shown in \cite{burgdorf}, $(\textrm{TP}_0)$ admits the following moment relaxation:
\begin{equation}\label{upop-trace1}
(\textrm{TP}):\quad
\begin{array}{rll}
\mu(f):=&\inf&L_{\y}(f)\\
&\textrm{s.t.}&M_{d}(\y)\succeq0,\\
&&M_{d}(\y)_{uv}=M_{d}(\y)_{wz},\quad\textrm{for all }u^{\star}v\stackrel{\textrm{cyc}}{\sim} w^{\star}z,\\
&&y_{1}=1.
\end{array}
\end{equation}
The dual of $(\textrm{TP})$ reads as:
\begin{equation}\label{upop-trace2}
(\textrm{TP})^*:\quad
\begin{array}{ll}
\sup&\mu\\
\textrm{s.t.} &\sum_{w\stackrel{\textrm{cyc}}{\sim} v}(\langle Q,A_{w}\rangle+\mu\delta_{1w})=\sum_{w\stackrel{\textrm{cyc}}{\sim} v}a_{w},\quad\forall v\in\W_{2d},\\
&Q\succeq0.
\end{array}
\end{equation}

Given $f=\sum_{w}a_ww\in\Sym\,\R\langle\underline{X}\rangle$ and $S=\{g_1,\ldots,g_m\}\subseteq\Sym\,\R\langle\underline{X}\rangle$, the trace minimization problem for $f$ over the semialgebraic set $\D_S$ is defined by:
\begin{equation}\label{cpop-trace}
(\textrm{TQ}_0):\quad\tr_{\min}(f, S):=\inf\{\tr\,f(\underline{A}):\underline{A}\in\D_S\}.
\end{equation}
We produce lower bounds on $\tr_{\min}(f, S)$ by restricting ourselves to a specific subset of ${D_S^\infty}$, 
obtained by considering the algebra of all bounded operators on a Hilbert space to finite von Neumann algebras~\cite{Takesaki03} of type I and type II. 
We introduce $\tr_{\min}(f,S)^{\II_1}$ as the trace minimum of $f$ on ${D_S^{\II_1}}$. Since ${D_S}$ can be described by ${D_S^{\II_1}} $, one has $\tr_{\min}(f,S)^{\II_1} \leq \tr_{\min}(f,S)$.
For a proper definition of ${D_S^{\II_1}}$, we refer the interested reader to, e.g.,~\cite[Definition~1.59]{burgdorf16}.
As shown in \cite{pironio}, one has the following series of moment relaxations indexed by $\hat{d}\ge d$ to obtain a hierarchy of lower bounds for $\tr_{\min}(f,S)^{\II_1}$:
\begin{equation}\label{cpop-trace1}
(\textrm{TQ}_{\hat{d}}):\quad
\begin{array}{rll}
\mu_{\hat{d}}(f, S):=&\inf &L_{\y}(f)\\
&\textrm{s.t.}&M_{\hat{d}}(\y)\succeq0,\\
&&M_{\hat{d}-d_j}(g_j\y)\succeq0,\quad j\in[m],\\
&&M_{\hat{d}}(\y)_{uv}=M_{\hat{d}}(\y)_{wz},\quad\textrm{for all }u^{\star}v\stackrel{\textrm{cyc}}{\sim} w^{\star}z,\\
&&y_{1}=1.
\end{array}
\end{equation}
We call $\hat{d}$ the {\em relaxation order}.
If the quadratic module $\M_S$ generated by $S$ is Archimedean then the sequence of bounds $(\mu_{\hat{d}}(f, S))_{\hat d\geq d}$ converges to $\tr_{\min}(f,S)^{\II_1}$. 
See, e.g.,~\cite[Corollary~3.5]{burgdorf16} for a proof.

The dual of \eqref{cpop-trace1} reads as:
\begin{equation}\label{cpop-trace2}
(\textrm{TQ}_{\hat{d}})^*:\quad
\begin{array}{rll}
&\sup&\mu\\
&\textrm{s.t.}&\sum_{w\stackrel{\textrm{cyc}}{\sim} v}(\sum_{j=0}^m\langle Q_j,D_{w}^j\rangle+\mu\delta_{1w})=\sum_{w\stackrel{\textrm{cyc}}{\sim} v}a_{w},\quad\forall v\in\W_{2\hat{d}},\\
&&Q_j\succeq0,\quad j\in\{0\}\cup[m].
\end{array}
\end{equation}

\subsection{Chordal graphs and sparse matrices}
In this subsection, we briefly revisit the relationship between chordal graphs and sparse matrices, which is crucial for the sparsity-exploitation of this paper. For more details on chordal graphs and sparse matrices, the reader is referred to \cite{va}.

An (undirected) {\em graph} $G(V,E)$ or simply $G$ consists of a set of nodes $V$ and a set of edges $E\subseteq\{\{v_i,v_j\}\mid (v_i,v_j)\in V\times V\}$. 
When $G$ is a graph, we also use $V(G)$ and $E(G)$ to indicate the node set of $G$ and the edge set of $G$, respectively. The {\em adjacency matrix} of $G$ is denoted by $B_G$ for which we put ones on its diagonal. For two graphs $G,H$, we say that $G$ is a {\em subgraph} of $H$ if $V(G)\subseteq V(H)$ and $E(G)\subseteq E(H)$, denoted by $G\subseteq H$. For a graph $G(V,E)$, a {\em cycle} of length $k$ is a set of nodes $\{v_1,v_2,\ldots,v_k\}\subseteq V$ with $\{v_k,v_1\}\in E$ and $\{v_i, v_{i+1}\}\in E$, for $i=1,\ldots,k-1$. A {\em chord} in a cycle $\{v_1,v_2,\ldots,v_k\}$ is an edge $\{v_i, v_j\}$ that joins two nonconsecutive nodes in the cycle. 

A graph is called a {\em chordal graph} if all its cycles of length at least four have a chord. Note that any non-chordal graph $G(V,E)$ can always be extended to a chordal graph $\overline{G}(V,\overline{E})$ by adding appropriate edges to $E$, which is called a {\em chordal extension} of $G(V,E)$. A {\em clique} $C\subseteq V$ of $G$ is a subset of nodes where $\{v_i,v_j\}\in E$ for any $v_i,v_j\in C$. If a clique $C$ is not a subset of any other clique, then it is called a {\em maximal clique}. It is known that maximal cliques of a chordal graph can be enumerated efficiently in linear time in the number of nodes and edges of the graph \cite{bp}.

Given a graph $G(V,E)$, a symmetric matrix $Q$ with row and column indices labeled by $V$ is said to have sparsity pattern $G$ if $Q_{\b\g}=Q_{\g\b}=0$ whenever $\b\ne\g$ and $\{\b,\g\}\notin E$, i.e., $B_G\circ Q=Q$. Let $\S_G$ be the set of symmetric matrices with sparsity pattern $G$.
A matrix in $\S_G$ exhibits a {\em block structure}. Each block corresponds to a maximal clique of $G$. The maximal block size is the maximal size of maximal cliques of $G$, namely, the \emph{clique number} of $G$. Note that there might be overlaps between blocks because different maximal cliques may share nodes.


Given a maximal clique $C$ of $G(V,E)$, we define a matrix $P_{C}\in \R^{|C|\times |V|}$ as
\begin{equation}\label{sec2-eq6}
[P_{C}]_{i\b}=\begin{cases}
1, &\textrm{if }C(i)=\b,\\
0, &\textrm{otherwise}.
\end{cases}
\end{equation}
where $C(i)$ denotes the $i$-th node in $C$, sorted in the ordering compatibly with $V$. Note that $Q_{C}=P_{C}QP_{C}^T\in \S^{|C|}$ extracts a principal submatrix $Q_C$ defined by the indices in the clique $C$ from a symmetry matrix $Q$, and $Q=P_{C}^TQ_{C}P_{C}$ inflates a $|C|\times|C|$ matrix $Q_{C}$ into a sparse $|V|\times |V|$ matrix $Q$.

The PSD matrices with sparsity pattern $G$ form a convex cone
\begin{equation}\label{sec2-eq5}
\S_+^{|V|}\cap\S_G=\{Q\in\S_G\mid Q\succeq0\}.
\end{equation}
When the sparsity pattern graph $G$ is chordal, the cone $\S_+^{|V|}\cap\S_G$ can be
decomposed as a sum of simple convex cones, as stated in the following theorem.
\begin{theorem}[\cite{va}, Theorem 9.2]\label{sec2-thm}
Let $G(V,E)$ be a chordal graph and assume that $C_1,\ldots,C_t$ are all the maximal cliques of $G(V,E)$. Then a matrix $Q\in\S_+^{|V|}\cap\S_G$ if and only if there exist $Q_{k}\in \S_+^{|C_k|}$ for $k=1,\ldots,t$ such that $Q=\sum_{k=1}^tP_{C_k}^TQ_{k}P_{C_k}$.
\end{theorem}

Given a graph $G(V,E)$, let $\Pi_{G}$ be the projection from $\S^{|V|}$ to the subspace $\S_G$, i.e., for $Q\in\S^{|V|}$,
\begin{equation}\label{sec2-eq7}
\Pi_{G}(Q)_{\b\g}=\begin{cases}
Q_{\b\g}, &\textrm{if }\{\b,\g\}\in E\textrm{ or }\b=\g,\\
0, &\textrm{otherwise}.
\end{cases}
\end{equation}

We denote by $\Pi_{G}(\S_+^{|V|})$ the set of matrices in $\S_G$ that have a PSD completion, i.e., 
\begin{equation}\label{sec2-eq8}
\Pi_{G}(\S_+^{|V|})=\{\Pi_{G}(Q)\mid Q\in\S_+^{|V|}\}.
\end{equation}

One can check that the PSD completable cone $\Pi_{G}(\S_+^{|V|})$ and the PSD cone $\S_+^{|V|}\cap\S_G$ form a pair of dual cones in $\S_G$; see \cite[Section 10.1]{va} for a proof. Moreover, for a chordal graph $G$, the decomposition result for the cone $\S_+^{|V|}\cap\S_G$ in Theorem \ref{sec2-thm} leads to the following characterization of the PSD completable cone $\Pi_{G}(\S_+^{|V|})$.
\begin{theorem}[\cite{va}, Theorem 10.1]\label{sec2-thm2}
Let $G(V,E)$ be a chordal graph and assume that $C_1,\ldots,C_t$ are all the maximal cliques of $G(V,E)$. Then a matrix $Q\in\Pi_{G}(\S_+^{|V|})$ if and only if $Q_{k}=P_{C_k}QP_{C_k}^T\succeq0$ for $k=1,\ldots,t$. Moreover, a matrix $Q\in\Pi_{G}(\S_{++}^{|V|})$ if and only if $Q_{k}=P_{C_k}QP_{C_k}^T\succ0$ for $k=1,\ldots,t$.
\end{theorem}

\section{Eigenvalue Optimization for Noncommutative Polynomials with term sparsity}\label{sec3}
In this section, we consider the eigenvalue optimization problem for noncommutative polynomials with term sparsity. For the reader's convenience, we first deal with the unconstrained case and then generalize to the constrained case.
\subsection{The unconstrained case}\label{ueo}
In this subsection, we describe an iterative procedure to exploit term sparsity for the moment-SOHS relaxations \eqref{upop-eigen1}-\eqref{upop-eigen2} of the unconstrained NCPOP $(\textrm{EP}_0)$ defined in \eqref{upop-eigen}.  

Let $f=\sum_{w\in\A}a_{w}w\in\Sym\,\R\langle\underline{X}\rangle$ with $\supp(f)=\A$ (w.l.o.g. assuming $1\in\A$). Assume that $\B$ is the monomial basis returned by the Newton chip method~\cite[\textsection2.3]{burgdorf16} with $r=|\B|$.
To represent the term sparsity in $f$, in the sequel we will consider graphs with $V:=\B$ as the set of nodes. Suppose that $G(V,E)$ is such a graph. We define the {\em support} of $G$ by
$$\supp(G):=\{u^{\star}v\mid(u,v)\in V\times V,\,\{u,v\}\in E\}.$$
We further define two operations on $G$: {\em support extension} and {\em chordal extension}. 

1) {\bf support extension}: The support extension of $G$, denoted by $\SE(G)$, is the graph with nodes $\B$ and with edges
$$E(\SE(G)):=\{\{u,v\}\mid(u,v)\in V\times V,\, u\ne v,\,u^{\star}v\in\supp(G)\cup\B^2\},$$
where $\B^2:=\{u^{\star}u\mid u\in\B\}$.

\begin{example}
Consider the following graph $G(V,E)$ with $$V=\{1,X,Y,Z,YZ,ZX,XY\} \textrm{ and } E=\{\{1,YZ\},\{Y,ZX\}\}.$$
Then $E(\SE(G))=\{\{1,YZ\},\{Y,ZX\},\{Y,Z\}\}$. See Figure \ref{support} for the support extension $\SE(G)$ of $G$.
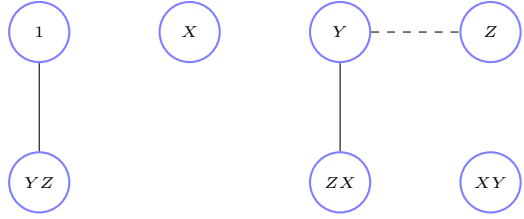
\begin{figure}[htbp]
\caption{The support extension $\SE(G)$ of $G$}\label{support}
\begin{center}
{\tiny
\begin{tikzpicture}[every node/.style={circle, draw=blue!50, thick, minimum size=8mm}]
\node (n1) at (0,0) {$1$};
\node (n2) at (2,0) {$X$};
\node (n3) at (4,0) {$Y$};
\node (n4) at (6,0) {$Z$};
\node (n5) at (0,-2) {$YZ$};
\node (n6) at (4,-2) {$ZX$};
\node (n7) at (6,-2) {$XY$};
\draw (n1)--(n5);
\draw (n3)--(n6);
\draw[dashed] (n3)--(n4);
\end{tikzpicture}}\\
{\small The dashed edges are added after support extension.}
\end{center}
\end{figure} 
\end{example}

2) {\bf chordal extension}: For a graph $G$, we denote any specific chordal extension of $G$ by $\overline{G}$. There are generally various chordal extensions of $G$. In this paper, we will consider two particular types of chordal extensions: {\em the maximal chordal extension} and {\em approximately minimum chordal extensions}. By the maximal chordal extension, we refer to the chordal extension that completes every connected component of $G$. The maximal chordal extension can be easily computed by listing all connected components. Another advantage of the maximal chordal extension is that there is no overlap among maximal cliques. However, the clique number of the maximal chordal extension may be large among all possible chordal extensions. A chordal extension with the lowest possible clique number is called a {\em minimum chordal extension}. Computing a minimum chordal extension of a graph is an NP-complete problem in general. Fortunately, several heuristic algorithms, e.g., the greedy minimum degree and the greedy minimum fill-ins, are known to efficiently produce a good approximation; see \cite{treewidth} for more detailed discussions. Throughout the paper, we assume that for graphs $G,H$,
\begin{equation}\label{assum}
G\subseteq H\Longrightarrow \overline{G}\subseteq\overline{H}.
\end{equation}
This assumption is reasonable since any chordal extension of $H$ restricting to $G$ is also a chordal extension of $G$.

\begin{example}
Consider the following graph $G(V,E)$ with $V=\{X_1,X_2,X_3,X_4,X_5,$ $X_6\}$ and $E=\{\{X_1,X_2\},\{X_2,X_3\},\{X_3,X_4\},\{X_4,X_5\},\{X_5,X_6\},\{X_6,X_1\}\}.$
See Figure \ref{chordal} for a minimum chordal extension $\overline{G}$ of $G$ which has $4$ maximal cliques of size $3$. On the other hand, the maximal chordal extension of $G$ has $1$ maximal clique of size $6$.
\begin{figure}[htbp]
\caption{A minimum chordal extension $\overline{G}$ of $G$}\label{chordal}
\begin{center}
{\tiny
\begin{tikzpicture}[every node/.style={circle, draw=blue!50, thick, minimum size=7.5mm}]
\node (n2) at (90:2) {$X_1$};
\node (n3) at (30:2) {$X_2$};
\node (n4) at (330:2) {$X_3$};
\node (n5) at (270:2) {$X_4$};
\node (n6) at (210:2) {$X_5$};
\node (n1) at (150:2) {$X_6$};
\draw (n2)--(n3);
\draw[dashed] (n2)--(n4);
\draw[dashed] (n2)--(n5);
\draw[dashed] (n2)--(n6);
\draw (n3)--(n4);
\draw (n4)--(n5);
\draw (n5)--(n6);
\draw (n6)--(n1);
\draw (n1)--(n2);
\end{tikzpicture}}\\
{\small The dashed edges are added after chordal extension.}
\end{center}
\end{figure}
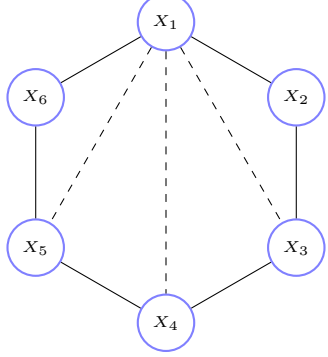 
\end{example}

Now we define $G_0(V,E_0)$ to be the graph with $V=\B$ and
\begin{equation}\label{e0}
    E_0=\{\{u,v\}\mid(u,v)\in V\times V,\,u\ne v,\,u^{\star}v\in\A\cup\B^2\},
\end{equation}
which is called the {\em term sparsity pattern (tsp) graph} associated with $f$. We then recursively define a sequence of graphs $(G_k(V,E_k))_{k\ge1}$ by alternately performing support extension and chordal extension to $G_0(V,E_0)$:
\begin{equation}\label{sec3-graph1}
G_k:=\overline{\SE(G_{k-1})}.
\end{equation}

When $f$ is sparse (i.e., $G_1$ is not complete), by replacing $M_{\B}(\y)\succeq0$ with the weaker condition $B_{G_k}\circ M_{\B}(\y)\in\Pi_{G_k}(\S_+^{r})$ in \eqref{upop-eigen1}, we obtain a series of sparse moment relaxations of $(\textrm{EP})$ (and $(\textrm{EP}_0)$) indexed by $k\ge1$:
\begin{equation}\label{upop-seigen1}
(\textrm{EP}^k):\quad
\begin{array}{rll}
\lambda_k(f):=&\inf &L_{\y}(f)\\
&\textrm{s.t.}&B_{G_k}\circ M_{\B}(\y)\in\Pi_{G_k}(\S^{r}_+),\\
&&y_{1}=1.
\end{array}
\end{equation}
We call $k$ the {\em sparse order}. By construction, one has $G_{k}\subseteq G_{k+1}$ for all $k\ge1$ and therefore the sequence of graphs $(G_k(V,E_k))_{k\ge1}$ stabilizes after a finite number of steps. We denote the stabilized graph by $G_{\circ}(V,E_{\circ})$ and the corresponding moment relaxation by $(\textrm{EP}^{\circ})$ (with optimum $\lambda_{\circ}(f)$).

For each $k\ge1$, the dual SDP of \eqref{upop-seigen1} reads as:
\begin{equation}\label{upop-seigen2}
(\textrm{EP}^k)^*:\quad
\begin{array}{ll}
\sup&\lambda\\
\textrm{s.t.}&\langle Q,A_{w}\rangle+\lambda\delta_{1w}=a_{w},\quad\forall w\in\supp(G_k)\cup\B^2,\\
&Q\in\S_+^{r}\cap\S_{G_k},
\end{array}
\end{equation}
where $A_{w}$ is defined in Section~\ref{sec2-eo}.

\begin{theorem}\label{sec3-thm}
Assume that $f\in\Sym\,\R\langle\underline{X}\rangle$. Then the followings hold:
\begin{enumerate}[(i)]
    \item For each $k\ge1$, there is no duality gap between $(\textrm{EP}^k)$ and $(\textrm{EP}^k)^*$.
    \item The sequence $(\lambda_k(f))_{k\ge 1}$ is monotone nondecreasing and $\lambda_k(f)\le\lambda_{\min}(f)$ for all $k$ (with $\lambda_{\min}(f)$ defined in \eqref{upop-eigen1}).
    \item If the maximal chordal extension is used in \eqref{sec3-graph1}, then $(\lambda_k(f))_{k\ge 1}$ converges to $\lambda_{\min}(f)$ in finitely many steps, i.e., $\lambda_{\circ}(f)=\lambda_{\min}(f)$.
\end{enumerate}
\end{theorem}
\begin{proof}
(i). Note that the SDP problem $(\textrm{EP})$ has a Slater's point, i.e., a strictly feasible solution (see, e.g., Proposition 4.9 in \cite{burgdorf}), say $M_{\B}(\y^*)$.
Since each block of $\Pi_{G_k}(M_{\B}(\y^*))$ is a principal submatrix of $M_{\B}(\y)$, we have that $\Pi_{G_k}(M_{\B}(\y^*))$ is a Slater's point of $(\textrm{EP}^k)$ by Theorem \ref{sec2-thm2}. So by the duality theory of convex programming, there is no duality gap between $(\textrm{EP}^k)$ and $(\textrm{EP}^k)^*$.

(ii). Because $G_{k}\subseteq G_{k+1}$, each maximal clique of $G_k$ is a subset of some maximal clique of $G_{k+1}$. Thus by Theorem \ref{sec2-thm2}, we have that $(\textrm{EP}^k)$ is a relaxation of $(\textrm{EP}^{k+1})$ (and also a relaxation of $(\textrm{EP})$). This yields the desired conclusions.

(iii). Let $\y^{*}=(y^{*}_{w})$ be an arbitrary feasible solution of $(\textrm{EP}^{\circ})$. Note that $\{y_{w}\mid w\in\supp(G_{\circ})\cup\B^2\}$ is the set of decision variables involved in $(\textrm{EP}^{\circ})$ and $\{y_{w}\mid w\in\B^{\star}\B\}$ is the set of decision variables involved in $(\textrm{EP})$.
We then define a vector $\overline{\y}^{*}=(\overline{y}^{*}_{w})_{w\in\B^{\star}\B}$ as follows:
$$\overline{y}_{w}^{*}=\begin{cases}y_{w}^{*},\quad\textrm{ if }w\in\supp(G_{\circ})\cup\B^2,\\
0,\quad\quad\textrm{otherwise}.
\end{cases}$$
If the maximal chordal extension is used in \eqref{sec3-graph1}, then matrices in $\Pi_{G_k}(\S^{r}_+)$ for all $k\ge1$ are block-diagonal (up to permutation). As a consequence, $B_{G_{k}}\circ M_{\B}(\y)\in\Pi_{G_k}(\S^{r}_+)$ implies $B_{G_{k}}\circ M_{\B}(\y)\succeq0$. By construction, we have $M_{\B}(\overline{\y}^*)=B_{G_{\circ}}\circ M_{\B}(\y^*)\succeq0$. Therefore $\overline{\y}^{*}$ is a feasible solution of $(\textrm{EP})$ and hence $L_{\y^{*}}(f)=L_{\overline{\y}^{*}}(f)\ge\lambda_{\min}(f)$.
This yields $\lambda_{\circ}(f)\ge\lambda_{\min}(f)$ since $\y^{*}$ is an arbitrary feasible solution of $(\textrm{EP}^{\circ})$. 
By (ii), we already have $\lambda_{\circ}(f)\le\lambda_{\min}(f)$. Therefore, $\lambda_{\circ}(f)=\lambda_{\min}(f)$.
\end{proof}

If (approximately) minimum chordal extensions are used in \eqref{sec3-graph1}, the sequence $(\lambda_k(f))_k$ doesn't necessarily converge to $\lambda_{\min}(f)$. The following is an example.
\begin{example}\label{ex1}
Consider the nc polynomial $f=X^2-XY-YX+3Y^2-2XYX+2XY^2X-YZ-ZY+6Z^2+9X^2Y+9Z^2Y-54ZYZ+142ZY^2Z$ (\cite{klep2019sparse}). The monomial basis given by the Newton chip method is $\{1,X,Y,Z,YX,YZ\}$. We have $E_0=\{\{1,YX\},\{1,YZ\},\{X,YX\},\{X,Y\},\{Y,Z\},\{Y,YZ\},\{Z,YZ\}\}$.
Figure \ref{tsp} shows the tsp graph $G_0$ (without dashed edges) and its chordal extension $G_1$ (with dashed edges) for $f$. The graph sequence $(G_k)_{k\ge1}$ immediately stabilizes at $k=1$. Solving the SDP problem ($\textrm{EP}^1$) associated with $G_1$, we obtain $\lambda_1(f)\approx-0.00355$ while we have $\lambda_{\min}(f)=0$. 

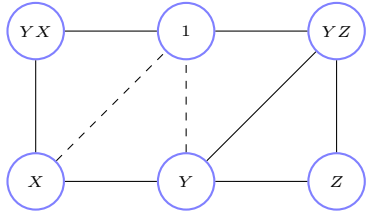
\begin{figure}[htbp]
\caption{The tsp graph $G_0$ and its chordal extension $G_1$ for Example \ref{ex1}}\label{tsp}
{\tiny
\begin{center}
\begin{tikzpicture}[every node/.style={circle, draw=blue!50, thick, minimum size=7.5mm}]
\node (n2) at (0,0) {$X$};
\node (n3) at (2,0) {$Y$};
\node (n4) at (4,0) {$Z$};
\node (n5) at (0,2) {$YX$};
\node (n1) at (2,2) {$1$};
\node (n6) at (4,2) {$YZ$};
\draw (n2)--(n3);
\draw (n3)--(n4);
\draw (n6)--(n4);
\draw (n1)--(n6);
\draw (n1)--(n5);
\draw (n2)--(n5);
\draw (n3)--(n6);
\draw[dashed] (n2)--(n1);
\draw[dashed] (n3)--(n1);
\end{tikzpicture}
\end{center}}
\end{figure}
\end{example}

The next result states that $\lambda_1(f)=\lambda_{\min}(f)$ always holds for a quadratic $f$.
\begin{theorem}\label{sec3-thm2}
Suppose that the nc polynomial $f\in\Sym\,\R\langle\underline{X}\rangle$ in ($\textrm{EP}_0$) is quadratic, i.e., $\deg(f)=2$. Then $\lambda_1(f)=\lambda_{\min}(f)$.
\end{theorem}
\begin{proof}
Assume $\supp(f)=\A$. Since $f$ is quadratic, we may take $\B=\{1,X_1,\ldots,X_n\}$ as a monomial basis. Let $G_0$ be the tsp graph associated with $f$.
We only need to prove that if $f$ admits a PSD Gram matrix, then $f$ admits a Gram matrix in $\S_+^{n+1}\cap\S_{G_0}$.
Suppose that $Q=[q_{ij}]_{i,j=0}^n$ is a PSD Gram matrix for $f$ indexed by $\B$. Note that for $i,j>0$, if $\{X_i,X_j\}\not\in E(G_0)$, then we must have $X_iX_j,X_jX_i\notin\A$, which implies $q_{ij}=0$; for $i=0,j>0$, if $\{1,X_j\}\not\in E(G_0)$, then we must have $X_j\notin\A$, which implies $q_{0j}=q_{j0}=0$. It follows that $Q\in\S_{G_0}$ as desired.
\end{proof}

\subsection{The constrained case}\label{sec3-con}
In this subsection, we generalize the iterative procedure in Section~\ref{ueo} to the constrained case and we show how to iteratively exploit term sparsity for the moment-SOHS hierarchy \eqref{cpop-eigen1}-\eqref{cpop-eigen2} of the constrained NCPOP $(\textrm{EQ}_0)$ defined in \eqref{cpop-eigen}. 

Assume that $f=\sum_{w}a_{w}w\in\Sym\,\R\langle\underline{X}\rangle$ and $S=\{g_1,\ldots,g_m\}\subseteq\Sym\,\R\langle\underline{X}\rangle$. Let
\begin{equation}\label{csupp}
\A = \supp(f)\cup\bigcup_{j=1}^m\supp(g_j).
\end{equation}
As in Section~\ref{sec2-eo}, we set $g_0:=1$ and let $d_j=\lceil\deg(g_j)/2\rceil$, $j\in\{0\}\cup[m]$ and $d=\max\{\lceil\deg(f)/2\rceil,d_1,\ldots,d_m\}$. Fixing a relaxation order $\hat{d}\ge d$, we define a graph $G_{\hat{d}}^{\textrm{tsp}}(V_{\hat{d}}^{\textrm{tsp}},E_{\hat{d}}^{\textrm{tsp}})$ with $V_{\hat{d}}^{\textrm{tsp}}=\W_{\hat{d}}$ and
\begin{equation}\label{sec3-eq1}
E_{\hat{d}}^{\textrm{tsp}}=\{\{u,v\}\mid(u,v)\in \W_{\hat{d}}\times \W_{\hat{d}},\,u\ne v,\,u^{\star}v\in\A\cup\W_{\hat{d}}^2\},
\end{equation}
where $\W_{\hat{d}}^2:=\{u^{\star}u\mid u\in\W_{\hat{d}}\}$. We call $G_{\hat{d}}^{\textrm{tsp}}$ the {\em term sparsity pattern (tsp) graph} associated with $\A$ (or $f$ and $S$). 

For a graph $G(V,E)$ with $V\subseteq\langle\underline{X}\rangle$ and $g\in\R\langle\underline{X}\rangle$, let us define
\begin{equation}
    \supp_{g}(G):=\{u^{\star}wv\mid(u,v)\in V\times V,\,\{u,v\}\in E, w\in\supp(g)\}.
\end{equation}
Let $G_{\hat{d},0}^{(0)}=G_{\hat{d}}^{\textrm{tsp}}$ and $G_{\hat{d},j}^{(0)}$ be an empty graph for $j\in[m]$. Then we recursively define a sequence of graphs $(G_{\hat{d},j}^{(k)}(V_{\hat{d},j},E_{\hat{d},j}^{(k)}))_{k\ge1}$ with $V_{\hat{d},j}=\W_{\hat{d}-d_j}$ for each $j\in\{0\}\cup[m]$ via two successive steps:\\
1) {\bf support extension}: Define $F_{\hat{d},j}^{(k)}$ to be the graph with $V(F_{\hat{d},j}^{(k)})=\W_{\hat{d}-d_j}$ and
\begin{align}\label{sec3-eq2}
E(F_{\hat{d},j}^{(k)})=&\{\{u,v\}\mid(u,v)\in\W_{\hat{d}-d_j}\times \W_{\hat{d}-d_j},\,u\ne v,\\
&u^{\star}\supp(g_j)v\cap(\bigcup_{j=0}^m\supp_{g_j}(G_{\hat{d},j}^{(k-1)})\cup\W_{\hat{d}}^2)\ne\emptyset\}.\notag
\end{align}
2) {\bf chordal extension}: Let
\begin{equation}\label{sec3-graph}
    G_{\hat{d},j}^{(k)}:=\overline{F_{\hat{d},j}^{(k)}}.
\end{equation}

Let $r_j=|\W_{\hat{d}-d_j}|$ for $j\in\{0\}\cup[m]$.
Then by replacing $M_{\hat{d}-d_j}(g_j\y)\succeq0$ with the weaker condition $B_{G_{\hat{d},j}^{(k)}}\circ M_{\hat{d}-d_j}(g_j\y)\in\Pi_{G_{\hat{d},j}^{(k)}}(\S_+^{r_j})$ for $j\in\{0\}\cup[m]$ in \eqref{cpop-eigen1}, we obtain the following series of sparse moment relaxations for ($\textrm{EQ}_{\hat{d}}$) indexed by $k\ge1$:
\begin{equation}\label{cpop-seigen1}
(\textrm{EQ}_{\hat{d},k}^{\textrm{ts}}):\quad
\begin{array}{rll}
\lambda^{\textrm{ts}}_{\hat{d},k}(f, S):=&\inf &L_{\y}(f)\\
&\textrm{s.t.}&B_{G_{\hat{d},0}^{(k)}}\circ M_{\hat{d}}(\y)\in\Pi_{G_{\hat{d},0}^{(k)}}(\S_+^{r_0}),\\
&&B_{G_{\hat{d},j}^{(k)}}\circ M_{\hat{d}-d_j}(g_j\y)\in\Pi_{G_{\hat{d},j}^{(k)}}(\S_+^{r_j}),\quad j\in[m],\\
&&y_{1}=1.
\end{array}
\end{equation}
We call $k$ the {\em sparse order}. By construction, one has $G_{\hat{d},j}^{(k)}\subseteq G_{\hat{d},j}^{(k+1)}$ for all $j,k$. Therefore, for every $j$, the sequence of graphs
$(G_{\hat{d},j}^{(k)})_{k\ge1}$ stabilizes after a finite number of steps. We denote the stabilized graphs by $G_{\hat{d},j}^{(\circ)}$ for all $j$ and denote the corresponding moment relaxation by $(\textrm{EQ}_{\hat{d},\circ}^{})$ (with optimum $\lambda_{\hat{d},\circ}^{\textrm{ts}}(f, S)$).

For each $k\ge1$, the dual of $(\textrm{EQ}_{\hat{d},k}^{\textrm{ts}})$ reads as:
\begin{equation}\label{cpop-seigen2}
(\textrm{EQ}_{\hat{d},k}^{\textrm{ts}})^*:
\begin{array}{ll}
\sup&\lambda\\
\textrm{s.t.}&\sum_{j=0}^m\langle Q_j,D_{w}^j\rangle+\lambda\delta_{1w}=a_{w},\forall w\in\bigcup_{j=0}^m\supp_{g_j}(G_{\hat{d},j}^{(k)}))\cup\W^2_{\hat{d}},\\
&Q_j\in\S_+^{r_j}\cap\S_{G_{\hat{d},j}^{(k)}},\quad j\in\{0\}\cup[m],
\end{array}
\end{equation}
where $\{D_{w}^j\}_{j,w}$ is defined in Section~\ref{sec2-eo}.


\begin{theorem}\label{cts-thm1}
Let $\{f\}\cup S\subseteq\Sym\,\R\langle\underline{X}\rangle$. Then the followings hold:
\begin{enumerate}[(i)]
    \item Assume that $S$ is feasible and contains a nc polynomial $g_1=R^2-\sum_{i=1}^nX_i^2$ for some $R>0$. Then for all $\hat{d},k$, there is no duality gap between $(\textrm{EQ}_{\hat{d},k}^{\textrm{ts}})$ and $(\textrm{EQ}_{\hat{d},k}^{\textrm{ts}})^*$.
    \item Fixing a relaxation order $\hat{d}\ge d$, the sequence $(\lambda^{\textrm{ts}}_{\hat{d},k}(f, S))_{k\ge1}$ is monotone nondecreasing and $\lambda^{\textrm{ts}}_{\hat{d},k}(f, S)\le\lambda_{\hat{d}}(f, S)$ for all $k$ (with $\lambda_{\hat{d}}(f, S)$ defined in \eqref{cpop-eigen1}).
    \item Fixing a sparse order $k\ge1$, the sequence $(\lambda^{\textrm{ts}}_{\hat{d},k}(f, S))_{\hat{d}\ge d}$ is monotone nondecreasing.
    \item If the maximal chordal extension is used in \eqref{sec3-graph}, then $(\lambda^{\textrm{ts}}_{\hat{d},k}(f, S))_{k\ge 1}$ converges to $\lambda_{\hat{d}}(f, S)$ in finitely many steps, i.e., $\lambda_{\hat{d},\circ}^{\textrm{ts}}(f, S)=\lambda_{\hat{d}}(f, S)$.
\end{enumerate}
\end{theorem}
\begin{proof}
(i). The proof proceeds in a similar manner as \cite{josz2016strong}.
First note that $(\textrm{EQ}_{\hat{d},k}^{\textrm{ts}})$ is feasible by considering the moments of the Dirac measure centred on some feasible point of $S$. Let $\CC=\bigcup_{i=0}^m\supp_{g_j}(G_{\hat{d},j}^{(k)})\cup\W_{\hat{d}}^2$. Consider a feasible solution $(y_{w})_{w\in\CC}$ of $(\textrm{EQ}_{\hat{d},k}^{\textrm{ts}})$ and extend it to $\y=(y_{w})_{w\in\W_{2\hat{d}}}$ by defining $y_w=0$ for $w\notin\CC$. 
Let $t\in\N$ be such that $1\le t\le\hat{d}$. Writing $g_1=\sum_{u}g_{1,u}u$, we have
\begin{align*}
    \Tr(M_{t-1}(g_1\y))&=\sum_{w\in\W_{t-1}}\sum_{u}g_{1,u}y_{w^{\star}uw}\\
    &=\sum_{w\in\W_{t-1}}(g_{1,1}y_{w^{\star}1w}+\sum_{i=1}^ng_{1,X_i^2}y_{w^{\star}X_i^2w})\\
    &=R^2\sum_{w\in\W_{t-1}}y_{w^{\star}w}-\sum_{w\in\W_{t-1}}\sum_{i=1}^ny_{w^{\star}X_i^2w}\\
    &=R^2\Tr(M_{t-1}(\y))+1-\Tr(M_{t}(\y)).
\end{align*}
Because $\Tr(M_{t-1}(g_1\y))\ge0$, we obtain $\Tr(M_{t}(\y))\le R^2\Tr(M_{t-1}(\y))+1$ and it follows $\Tr(M_{\hat{d}}(\y))\le\sum_{t=0}^{\hat{d}}R^{2t}$. Since $B_{G_{\hat{d},0}^{(k)}}\circ M_{\hat{d}}(\y)\in\Pi_{G_{\hat{d},0}^{(k)}}(\S_+^{r_0})$, there exists a PSD matrix $P\in\S_+^{r_0}$ such that $B_{G_{\hat{d},0}^{(k)}}\circ M_{\hat{d}}(\y)=B_{G_{\hat{d},0}^{(k)}}\circ P$. We have $\Tr((B_{G_{\hat{d},0}^{(k)}}\circ M_{\hat{d}}(\y))^2)\le\Tr(P^2)\le\Tr(P)^2=\Tr(M_{\hat{d}}(\y))^2$. From this we deduce that $$\sqrt{\sum_{w\in\CC}y_w^2}\le\sqrt{\Tr((B_{G_{\hat{d},0}^{(k)}}\circ M_{\hat{d}}(\y))^2)}\le\Tr(M_{\hat{d}}(\y))\le\sum_{t=0}^{\hat{d}}R^{2t}.$$ Then the conclusion follows from the same argument as for Theorem 1 in \cite{josz2016strong}.

(ii). For all $j,k$, because $G_{\hat{d},j}^{(k)}\subseteq G_{\hat{d},j}^{(k+1)}$, each maximal clique of $G_{\hat{d},j}^{(k)}$ is a subset of some maximal clique of $G_{\hat{d},j}^{(k+1)}$. Hence by Theorem \ref{sec2-thm2}, $(\textrm{EQ}_{\hat{d},k}^{\textrm{ts}})$ is a relaxation of $(\textrm{EQ}_{\hat{d},k+1}^{\textrm{ts}})$ (and also a relaxation of $(\textrm{EQ}_{\hat{d}})$). Therefore, $(\lambda^{\textrm{ts}}_{\hat{d},k}(f, S))_{k\ge1}$ is monotone nondecreasing and $\lambda^{\textrm{ts}}_{\hat{d},k}(f, S)\le\lambda_{\hat{d}}(f, S)$ for all $k$.

(iii). The conclusion follows if we can show that $G_{\hat{d},j}^{(k)}\subseteq G_{\hat{d}+1,j}^{(k)}$ for all $\hat{d},j$ since by Theorem \ref{sec2-thm2} this implies that $(\textrm{EQ}_{\hat{d},k}^{\textrm{ts}})$ is a relaxation of $(\textrm{EQ}_{\hat{d}+1,k}^{\textrm{ts}})$. Let us prove $G_{\hat{d},j}^{(k)}\subseteq G_{\hat{d}+1,j}^{(k)}$ by induction on $k$. For $k=1$, from $\eqref{sec3-eq1}$, we have $E_{\hat{d},0}^{(0)}\subseteq E_{\hat{d}+1,0}^{(0)}$, which implies that $G_{\hat{d},j}^{(1)}\subseteq G_{\hat{d}+1,j}^{(1)}$ for all $\hat{d},j$. Now assume that $G_{\hat{d},j}^{(k)}\subseteq G_{\hat{d}+1,j}^{(k)}$ for all $\hat{d},j$ hold for a given $k\ge 1$. Then from \eqref{assum}, \eqref{sec3-eq2}, \eqref{sec3-graph} and by the induction hypothesis, we have $G_{\hat{d},j}^{(k+1)}\subseteq G_{\hat{d}+1,j}^{(k+1)}$ for all $\hat{d},j$, which completes the induction and also completes the proof.

(iv). Let $\y^{*}=(y^{*}_{w})$ be an arbitrary feasible solution of $(\textrm{EQ}_{\hat{d},\circ}^{\textrm{ts}})$. Note that $\{y_{w}\mid w\in\bigcup_{i=0}^m\supp_{g_j}(G_{\hat{d},j}^{(\circ)})\cup\W_{\hat{d}}^2\}$ is the set of decision variables involved in $(\textrm{EQ}_{\hat{d},\circ}^{\textrm{ts}})$ and $\{y_{w}\mid w\in\W_{\hat{d}}^{\star}\W_{\hat{d}}\}$ is the set of decision variables involved in ($\textrm{EQ}_{\hat{d}}$).
We then define a vector $\overline{\y}^{*}=(\overline{y}^{*}_{w})_{w\in\W_{\hat{d}}^{\star}\W_{\hat{d}}}$ as follows:
$$\overline{y}_{w}^{*}=\begin{cases}y_{w}^{*},\quad\textrm{ if }w\in\bigcup_{i=0}^m\supp_{g_j}(G_{\hat{d},j}^{(\circ)})\cup\W_{\hat{d}}^2,\\
0,\quad\quad\textrm{otherwise}.
\end{cases}$$
If the maximal chordal extension is used in \eqref{sec3-graph}, then the matrices in $\Pi_{G_{\hat{d},j}^{(k)}}(\S_+^{r_j})$ for all $k\ge1$ are block-diagonal (up to permutation). As a consequence, $B_{G_{\hat{d},j}^{(k)}}\circ M_{\hat{d}-d_j}(g_j\y)\in\Pi_{G_{\hat{d},j}^{(k)}}(\S_+^{r_j})$ implies $B_{G_{\hat{d},j}^{(k)}}\circ M_{\hat{d}-d_j}(g_j\y)\succeq0$. 
By construction, we have $M_{\hat{d}-d_j}(g_j\overline{\y}^*)=B_{G_{\hat{d},j}^{(\circ)}}\circ M_{\hat{d}-d_j}(g_j\y^*)\succeq0$ for all $j\in\{0\}\cup[m]$. Therefore $\overline{\y}^{*}$ is a feasible solution of ($\textrm{EQ}_{\hat{d}}$) and hence $L_{\y^{*}}(f)=L_{\overline{\y}^{*}}(f)\ge\lambda_{\hat{d}}(f, S)$,  which yields $\lambda_{\hat{d},\circ}^{\textrm{ts}}(f, S)\ge\lambda_{\hat{d}}(f, S)$ since $\y^{*}$ is an arbitrary feasible solution of $(\textrm{EQ}_{\hat{d},\circ}^{\textrm{ts}})$. 
By (ii), we already have $\lambda_{\hat{d},\circ}^{\textrm{ts}}(f, S)\le\lambda_{\hat{d}}(f, S)$. Therefore, $\lambda_{\hat{d},\circ}^{\textrm{ts}}(f, S)=\lambda_{\hat{d}}(f, S)$.
\end{proof}

Following from Theorem \ref{cts-thm1}, we have the following two-level hierarchy of lower bounds for the optimum $\lambda_{\min}(f,S)$ of $(\textrm{EQ}_0)$:
\begin{equation}\label{cliquehierc}
\begin{matrix}
\lambda^{\textrm{ts}}_{d,1}(f, S)&\le&\lambda^{\textrm{ts}}_{d,2}(f, S)&\le&\cdots&\le&\lambda_{d}(f, S)\\
\vge&&\vge&&&&\vge\\
\lambda^{\textrm{ts}}_{d+1,1}(f, S)&\le&\lambda^{\textrm{ts}}_{d+1,2}(f, S)&\le&\cdots&\le&\lambda_{d+1}(f, S)\\
\vge&&\vge&&&&\vge\\
\vdots&&\vdots&&\vdots&&\vdots\\
\vge&&\vge&&&&\vge\\
\lambda^{\textrm{ts}}_{\hat{d},1}(f, S)&\le&\lambda^{\textrm{ts}}_{\hat{d},2}(f, S)&\le&\cdots&\le&\lambda_{\hat{d}}(f, S)\\
\vge&&\vge&&&&\vge\\
\vdots&&\vdots&&\vdots&&\vdots\\
\end{matrix}
\end{equation}
We call the array of lower bounds \eqref{cliquehierc} (and its corresponding moment-SOHS relaxations \eqref{cpop-seigen1}-\eqref{cpop-seigen2}) the NCTSSOS hierarchy associated with $(\textrm{EQ}_0)$.

\begin{remark}
The NCTSSOS hierarchy entails a trade-off between the computational cost and the quality of the obtained lower bound via the two parameters $\hat{d}$ and $k$. Besides, one has the freedom to choose a specific chordal extension for any graph involved in \eqref{sec3-graph} (e.g., the maximal chordal extension, approximately minimum chordal extension and so on). This choice affects the resulting sizes of (submatrix) blocks and the quality of the lower bound given by the corresponding SDP relaxation. Intuitively, chordal extensions with smaller clique numbers should lead to (submatrix) blocks of smaller sizes and lower bounds of (possibly) lower quality while chordal extensions with larger clique numbers should lead to (submatrix) blocks with larger sizes and lower bounds of (possibly) higher quality.
\end{remark}

\begin{example}\label{sec3-ex4}
Consider $f=2-X^2+XY^2X-Y^2$ and $S=\{4-X^2-Y^2,XY+YX-2\}$. We draw the tsp graph $G_{2,0}^{(0)}$ for $f$ and $S$ in Figure \ref{ex4-1}. Since $G_{2,0}^{(0)}$ is already a chordal graph, we don't need any chordal extension. Hence $G_{2,0}^{(1)}=G_{2,0}^{(0)}$ and $(G_{2,j}^{(k)})_{k\ge1}$ immediately stabilizes at $k=1$ for all $j$. We compute that $\lambda^{\textrm{ts}}_{2,1}(f, S)=\lambda_{\min}(f, S)=-1$.

\begin{figure}[htbp]
\caption{The tsp graph for Example \ref{sec3-ex4}}\label{ex4-1}
\begin{center}
{\tiny
\begin{tikzpicture}[every node/.style={circle, draw=blue!50, thick, minimum size=8mm}]
\node (n1) at (90:2) {$1$};
\node (n2) at (162:2) {$X^2$};
\node (n3) at (234:2) {$XY$};
\node (n4) at (306:2) {$YX$};
\node (n5) at (18:2) {$Y^2$};
\draw (n3)--(n1);
\draw (n4)--(n1);
\draw (n5)--(n1);
\draw (n1)--(n2);
\node[xshift=120] (n6) at (90:1) {$X$};
\node[xshift=120] (n7) at (270:1) {$Y$};
\draw (n6)--(n7);
\end{tikzpicture}}
\end{center}
\end{figure}
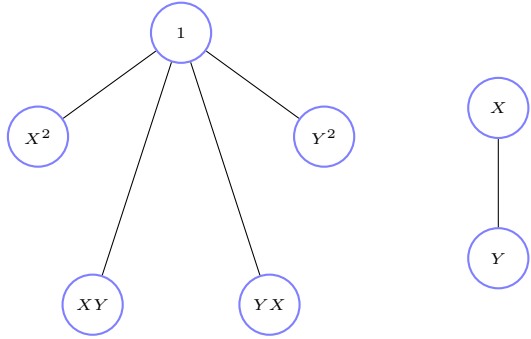 
\end{example}

\section{Eigenvalue optimization for noncommutative polynomials with combined correlative-term sparsity}\label{eo-cts}
The exploitation of term sparsity developed in the previous section can be combined with the exploitation of correlative sparsity discussed in \cite{klep2019sparse} to reduce the computational cost further. To begin with, let us recall some basics on correlative sparsity. For more details, the reader is referred to \cite{klep2019sparse}.

\subsection{Eigenvalue optimization for noncommutative polynomials with correlative sparsity}\label{cs}
As in the commutative case, the exploitation of correlative sparsity in the moment-SOHS hierarchy for NCPOPs consists of two steps: 1) partition the set of variables into subsets according to the correlations between variables emerging in the problem, and 2) construct a sparse moment-SOHS hierarchy with respect to the former partition of variables \cite{klep2019sparse}.

More concretely, assuming $f=\sum_{w}a_{w}w\in\Sym\,\R\langle\underline{X}\rangle$ and $S=\{g_1,\ldots,g_m\}\subseteq\Sym\,\R\langle\underline{X}\rangle$, we define the {\em correlative sparsity pattern (csp) graph} associated with $f$ and $S$ to be the graph $G^{\textrm{csp}}$ with nodes $V=[n]$ and edges $E$ satisfying $\{i,j\}\in E$ if one of followings holds:
\begin{enumerate}
    \item[(i)] there exists $w\in\supp(f)\textrm{ s.t. }X_i,X_j\in\var(w)$;
    \item[(ii)] there exists $k$, with $1\le k\le m, \textrm{ s.t. }X_i,X_j\in\var(g_k)$,
\end{enumerate}
where we use $\var(g)$ to denote the set of variables effectively involved in $g\in\R\langle\underline{X}\rangle$. Let $\overline{G}^{\textrm{csp}}$ be a chordal extension of $G^{\textrm{csp}}$ and $I_l,l\in[p]$ be the maximal cliques of $\overline{G}^{\textrm{csp}}$ with cardinal denoted by $n_l,l\in[p]$. Let $\R\langle\underline{X}(I_l)\rangle$ denote the ring of nc polynomials in the $n_l$ variables $\underline{X}(I_l) = \{X_i\mid i\in I_l\}$. We then partition the constraints $g_1,\ldots,g_m$ into groups $\{g_j\mid j\in J_l\}, l\in[p]$ which satisfy:
\begin{enumerate}
    \item[(i)] $J_1,\ldots,J_p\subseteq[m]$ are pairwise disjoint and $\bigcup_{l=1}^pJ_l=[m]$;
    \item[(ii)] for any $j\in J_l$, $\var(g_j)\subseteq \underline{X}(I_l)$, $l\in[p]$.
\end{enumerate}

Next, with $l\in[p]$ fixed, $d$ a positive integer and $g\in\R\langle\underline{X}(I_l)\rangle$, let $M_d(\y, I_l)$ (resp. $M_d(g\y, I_l)$)
be the moment (resp. localizing) submatrix obtained from $M_d(\y)$ (resp. $M_d(g\y)$) by retaining only those rows (and columns) indexed by $w\in\langle\underline{X}(I_l)\rangle$ of $M_d(\y)$ (resp. $M_d(g\y)$).

Then with $\hat{d}\ge d:= \max\{\lceil\deg(f)/2\rceil,\lceil\deg(g_1)/2\rceil,\ldots,\lceil\deg(g_m)/2\rceil\}$, the moment SDP relaxation for $(\textrm{EQ}_0)$ based on correlative sparsity is defined as:
\begin{equation}\label{seccs-eq1}
(\textrm{EQ}^{\textrm{cs}}_{\hat{d}}):\quad
\begin{array}{rll}
\lambda^{\textrm{cs}}_{\hat{d}}(f, S):=
&\inf &L_{\y}(f)\\
&\textrm{s.t.}&M_{\hat{d}}(\y, I_l)\succeq0,\quad l\in[p],\\
&&M_{\hat{d}-d_j}(g_j\y, I_l)\succeq0,\quad j\in J_l, l\in[p],\\
&&y_1=1.
\end{array}
\end{equation}

\begin{remark}
As shown in \cite{klep2019sparse} under some Archimedean's condition (slightly stronger than compactness), the sequence $(\lambda^{\textrm{cs}}_{\hat{d}}(f, S))_{\hat{d}\ge d}$ converges to the global optimum $\lambda_{\min}(f, S)$.
\end{remark}

\subsection{Eigenvalue optimization for noncommutative polynomials with combined correlative-term sparsity}\label{cts}
The combination of correlative sparsity and term sparsity proceeds in a similar manner as for the commutative case in \cite{wang4}.
Assume that $f=\sum_{w}a_{w}w\in\Sym\,\R\langle\underline{X}\rangle$ and $S=\{g_1,\ldots,g_m\}\subseteq\Sym\,\R\langle\underline{X}\rangle$, $G^{\textrm{csp}}$ is the csp graph associated with $f$ and $S$, and $\overline{G}^{\textrm{csp}}$ is a chordal extension of $G^{\textrm{csp}}$. Let $I_l,l\in[p]$ be the maximal cliques of $\overline{G}^{\textrm{csp}}$ with cardinal denoted by $n_l,l\in[p]$. Then the set of variables $\underline{X}$ is partitioned into $\underline{X}(I_1), \underline{X}(I_2), \ldots, \underline{X}(I_p)$. Let $J_1,\ldots,J_p$ be defined as in Section~\ref{cs}.

Now we consider the term sparsity pattern for each subsystem involving the variables $\underline{X}(I_l)$, $l\in[p]$ respectively as follows. Let
\begin{equation}\label{cts-eq1}
    \A:= \supp(f)\cup\bigcup_{j=1}^m\supp(g_j)\textrm{ and }
    \A_l:= \{w\in\A\mid\var(w)\subseteq \underline{X}(I_l)\} \,,
\end{equation}
for $l\in[p]$. As before, let $g_0=1$, $d_j=\lceil\deg(g_j)/2\rceil$, $j\in\{0\}\cup[m]$ and $d=\max\{\lceil\deg(f)/2\rceil,d_1,\ldots,d_m\}$. Fix a relaxation order $\hat{d}\ge d$. Let $\W_{\hat{d}-d_j,l}$ be the standard monomial basis of degree $\le\hat{d}-d_j$ with respect to the variables $\underline{X}(I_l)$ and $G_{\hat{d},l}^{\textrm{tsp}}$ be the tsp graph with nodes $\W_{\hat{d},l}$ associated with $\A_l$ defined as in Section~\ref{sec3-con}.
Assume that $G_{\hat{d},l,0}^{(0)}=G_{\hat{d},l}^{\textrm{tsp}}$ and $G_{\hat{d},l,j}^{(0)},j\in J_l, l\in[p]$ are empty graphs. Letting
\begin{equation}\label{cts-eq2}
    \CC_{\hat{d}}^{(k-1)}:=\bigcup_{l=1}^p\bigcup_{j\in \{0\}\cup J_l}\supp_{g_j}(G_{\hat{d},l,j}^{(k-1)})\cup\W_{\hat{d}}^2,\quad k\ge1,
\end{equation}
we recursively define a sequence of graphs $(G_{\hat{d},l,j}^{(k)}(V_{\hat{d},l,j},E_{\hat{d},l,j}^{(k)}))_{k\ge1}$ with $V_{\hat{d},l,j}=\W_{\hat{d}-d_j,l}$ for $j\in\{0\}\cup J_l,l\in[p]$ by
\begin{equation}\label{cts-eq3}
G_{\hat{d},l,j}^{(k)}:=\overline{F_{\hat{d},l,j}^{(k)}},
\end{equation}
where $F_{\hat{d},l,j}^{(k)}$ is the graph with $V(F_{\hat{d},l,j}^{(k)})=\W_{\hat{d}-d_j,l}$ and
\begin{equation}\label{cts-eq4}
E(F_{\hat{d},l,j}^{(k)})=\{\{u,v\}\mid(u,v)\in\W_{\hat{d}-d_j,l}\times \W_{\hat{d}-d_j,l}, u^{\star}\supp(g_j)v\cap\CC_{\hat{d}}^{(k-1)}\ne\emptyset\}.
\end{equation}

Let $r_{l,j}=|\W_{\hat{d}-d_j,l}|$ for all $l,j$. Then for each $k\ge1$, the sparse moment relaxation based on combined correlative-term sparsity for $(\textrm{EQ}_0)$ is defined as:
\begin{equation}\label{cts-eq5}
(\textrm{EQ}^{\textrm{cs-ts}}_{\hat{d},k}):
\begin{array}{rll}
\lambda^{\textrm{cs-ts}}_{\hat{d},k}(f, S):=
&\inf&L_{\y}(f)\\
&\textrm{s.t.}&B_{G_{\hat{d},l,0}^{(k)}}\circ M_{\hat{d}}(\y, I_l)\in\Pi_{G_{\hat{d},l,0}^{(k)}}(\S_+^{r_{l,0}}), l\in[p],\\
&&B_{G_{\hat{d},l,j}^{(k)}}\circ M_{\hat{d}-d_j}(g_j\y, I_l)\in\Pi_{G_{\hat{d},l,j}^{(k)}}(\S_+^{r_{l,j}}), j\in J_l,l\in[p],\\
&&y_1=1.
\end{array}
\end{equation}

For any $l,j$, write $M_{\hat{d}-d_j}(g_j\y, I_l)=\sum_{w}D_{w}^{l,j}y_{w}$ for appropriate matrices $\{D_{w}^{l,j}\}$. Then for each $k\ge1$, the dual of $(\textrm{EQ}^{\textrm{cs-ts}}_{\hat{d},k})$ reads as:
\begin{equation}\label{cts-eq6}
(\textrm{EQ}^{\textrm{cs-ts}}_{\hat{d},k})^*:\quad
\begin{cases}
\sup\,&\lambda\\
\textrm{s.t.}\, &\sum_{l=1}^p\sum_{j\in \{0\}\cup J_l}\langle Q_{l,j},D_{w}^{l,j}\rangle+\lambda\delta_{1w}=a_{w},\quad\forall w\in\CC_{\hat{d}}^{(k)},\\
&Q_{l,j}\in\S_+^{r_{l,j}}\cap\S_{G_{\hat{d},l,j}^{(k)}},\quad j\in \{0\}\cup J_l, l\in[p],
\end{cases}
\end{equation}
where $\CC_{\hat{d}}^{(k)}$ is defined as in \eqref{cts-eq2}.

By similar arguments as for Theorem \ref{cts-thm1}, we can prove the following theorem.
\begin{theorem}\label{cts-prop1}
Assume that $\{f\}\cup S\subseteq\Sym\,\R\langle\underline{X}\rangle$. Then the followings hold:
\begin{enumerate}[(i)]
    \item Fixing a relaxation order $\hat{d}\ge d$, the sequence $(\lambda^{\textrm{cs-ts}}_{\hat{d},k}(f, S))_{k\ge1}$ is monotone non-decreasing and $\lambda^{\textrm{cs-ts}}_{\hat{d},k}(f, S)\le\lambda^{\textrm{cs}}_{\hat{d}}(f, S)$ for all $k$ (with $\lambda^{\textrm{cs}}_{\hat{d}}(f, S)$ defined in Section~\ref{cs}).
    \item Fixing a sparse order $k\ge 1$, the sequence $(\lambda^{\textrm{cs-ts}}_{\hat{d},k}(f, S))_{\hat{d}\ge d}$ is monotone non-decreasing.
    \item If the maximal chordal extension is used in \eqref{cts-eq3}, then $(\lambda^{\textrm{cs-ts}}_{\hat{d},k}(f, S))_{k\ge 1}$ converges to $\lambda^{\textrm{cs}}_{\hat{d}}(f, S)$ in finitely many steps.
\end{enumerate}
\end{theorem}

From Theorem \ref{cts-prop1}, we deduce the following two-level hierarchy of lower bounds for the optimum $\lambda_{\min}(f, S)$ of $(\textrm{EQ}_{0})$:
\begin{equation}\label{mixhierc}
\begin{matrix}
\lambda^{\textrm{cs-ts}}_{d,1}(f, S)&\le&\lambda^{\textrm{cs-ts}}_{d,2}(f, S)&\le&\cdots&\le&\lambda^{\textrm{cs}}_{d}(f, S)\\
\vge&&\vge&&&&\vge\\
\lambda^{\textrm{cs-ts}}_{d+1,1}(f, S)&\le&\lambda^{\textrm{cs-ts}}_{d+1,2}(f, S)&\le&\cdots&\le&\lambda^{\textrm{cs}}_{d+1}(f, S)\\
\vge&&\vge&&&&\vge\\
\vdots&&\vdots&&\vdots&&\vdots\\
\vge&&\vge&&&&\vge\\
\lambda^{\textrm{cs-ts}}_{\hat{d},1}(f, S)&\le&\lambda^{\textrm{cs-ts}}_{\hat{d},2}(f, S)&\le&\cdots&\le&\lambda^{\textrm{cs}}_{\hat{d}}(f, S)\\
\vge&&\vge&&&&\vge\\
\vdots&&\vdots&&\vdots&&\vdots\\
\end{matrix}
\end{equation}

\section{Trace Optimization for Noncommutative Polynomials with term sparsity}\label{sec5}
The results presented in the previous sections concerning eigenvalue optimization for noncommutative polynomials with term sparsity can be slightly adjusted to deal with trace optimization for noncommutative polynomials with term sparsity. We present the main results concerning trace optimization in this section and omit the proofs. 
\subsection{The unconstrained case}
Let $f=\sum_{w\in\A}a_{w}w\in\Sym\,\R\langle\underline{X}\rangle$ with $\supp(f)=\A$ (w.l.o.g. assuming $1\in\A$) and let $d=\cdeg(f)$. We define $H_0(V,E_0)$ to be the graph with $V=\W_d$ and
\begin{equation}\label{sec4-eq0}
    E_0=\{\{u,v\}\mid(u,v)\in V\times V,\,u\ne v,\,[u^{\star}v]\in[\A\cup\W_d^2]\}.
\end{equation}
We recursively define a sequence of graphs $(H_k(V,F_k))_{k\ge1}$ by
\begin{equation}\label{sec4-graph1}
H_k:=\overline{\CSE(H_{k-1})},
\end{equation}
where $\CSE(H_{k-1})$ (the cyclic support extension of $H_{k-1}$) is the graph with nodes $\W_d$ and with edges
$$E(\CSE(H_{k-1})):=\{\{u,v\}\mid(u,v)\in V\times V,\, u\ne v,\,[u^{\star}v]\in[\supp(H_{k-1})\cup\W_d^2]\}.$$

Let $r=|\W_d|$. As for eigenvalue optimization, we can consider the following series of sparse moment relaxations for $(\textrm{TP})$ indexed by $k\ge1$:
\begin{equation}\label{upop-strace1}
(\textrm{TP}^k):\quad
\begin{array}{rll}
\mu_k(f):=&\inf&L_{\y}(f)\\
&\textrm{s.t.}&B_{H_k}\circ M_{d}(\y)\in\Pi_{H_k}(\S^{r}_+),\\
&&[B_{H_k}\circ M_{d}(\y)]_{uv}=[B_{H_k}\circ M_{d}(\y)]_{wz},\textrm{for all }u^{\star}v\stackrel{\textrm{cyc}}{\sim} w^{\star}z,\\
&&y_{1}=1.
\end{array}
\end{equation}
The dual of $(\textrm{TP}^k)$ reads as:
\begin{equation}\label{upop-strace2}
(\textrm{TP}^k)^*:\quad
\begin{array}{ll}
\sup&\mu\\
\textrm{s.t.} &\sum_{w\stackrel{\textrm{cyc}}{\sim} v}(\langle Q,A_{w}\rangle+\mu\delta_{1w})=\sum_{w\stackrel{\textrm{cyc}}{\sim} v}a_{w},\quad\forall v\in[\supp(H_k)\cup\W_d^2],\\
&Q\in\S_+^{r}\cap\S_{H_k},
\end{array}
\end{equation}
where $A_{w}$ is defined as in Section~\ref{sec2-eo}. We call $k$ the {\em sparse order}. There is no duality gap between $(\textrm{TP}^k)$ and $(\textrm{TP}^k)^*$. By construction, one has $H_{k}\subseteq H_{k+1}$ for all $k\ge1$ and therefore the sequence of graphs $(H_k(V,E_k))_{k\ge1}$ stabilizes after a finite number of steps. We denote the stabilized graph by $H_{\circ}(V,E_{\circ})$ and the optimum of the corresponding SDP relaxation by $\mu_{\circ}(f)$. 

As for eigenvalue optimization, we obtain the following hierarchy of lower bounds for $\tr_{\min}(f)$:
\begin{equation}\label{cliquehier_tr}
\mu_1(f)\le\mu_2(f)\le\cdots\le\mu_{\circ}(f)\le\mu(f)\le\tr_{\min}(f).
\end{equation}
Moreover, if the maximal chordal extension is used in \eqref{sec4-graph1}, then $(\mu_k(f))_{k\ge 1}$ converges to $\mu(f)$ in finitely many steps, i.e., $\mu_{\circ}(f)=\mu(f)$.

\begin{remark}
The monomial basis $\W_d$ used in this subsection can be replaced by the reduced monomial basis returned by the tracial Newton polytope method \cite[\textsection3.3]{burgdorf16}. However, for the numerical experiments performed in this paper (see Section~\ref{sec:benchs}), we have noticed that it is somewhat expensive to implement the tracial Newton polytope method while not yielding a significant reduction of the size of the monomial basis. Hence we stick to the standard monomial basis $\W_d$.
\end{remark}

\subsection{The constrained case}\label{trace-con}
Assume that $f=\sum_{w}a_{w}w\in\Sym\,\R\langle\underline{X}\rangle$ and $S=\{g_1,\ldots,g_m\}\subseteq\Sym\,\R\langle\underline{X}\rangle$. As before, let $\A=\supp(f)\cup\bigcup_{j=1}^m\supp(g_j)$, $g_0=1$ and $d_j=\lceil\deg(g_j)/2\rceil$, $j\in\{0\}\cup[m]$. Let $d=\max\{\lceil\cdeg(f)/2\rceil,d_1,\ldots,d_m\}$. Fix a relaxation order $\hat{d}\ge d$. We define a graph $H_{\hat{d}}^{\textrm{tsp}}(V_{\hat{d}},E_{\hat{d}}^{\textrm{tsp}})$ with $V_{\hat{d}}=\W_{\hat{d}}$ and
\begin{equation}\label{sec4-eq1}
E_{\hat{d}}^{\textrm{tsp}}=\{\{u,v\}\mid(u,v)\in V_{\hat{d}}\times V_{\hat{d}},\,u\ne v,\,[u^{\star}v]\in[\A\cup\W_{\hat{d}}^2]\},
\end{equation}
which is called the {\em cyclic tsp graph} associated with $\A$ (or $f$ and $S$). Let $H_{\hat{d},0}^{(0)}=H_{\hat{d}}^{\textrm{tsp}}$ and $H_{\hat{d},j}^{(0)}$ be an empty graph for $j\in[m]$. We recursively define a sequence of graphs $(H_{\hat{d},j}^{(k)}(V_{\hat{d},j},E_{\hat{d},j}^{(k)}))_{k\ge1}$ with $V_{\hat{d},j}=\W_{\hat{d}-d_j}$ for $j\in\{0\}\cup[m]$ via two successive steps:\\
1) {\bf cyclic support extension}: Define $K_{\hat{d},j}^{(k)}$ to be the graph with $V(K_{j,\hat{d}}^{(k)})=\W_{\hat{d}-d_j}$ and
\begin{align}\label{sec4-eq2}
E(K_{\hat{d},j}^{(k)})=&\{\{u,v\}\mid(u,v)\in\W_{\hat{d}-d_j}\times \W_{\hat{d}-d_j},\,u\ne v,\\
&[u^{\star}\supp(g_j)v]\cap[\bigcup_{j=0}^m\supp_{g_j}(H_{j,\hat{d}}^{(k-1)})\cup\W_{\hat{d}}^2]\ne\emptyset\}.\notag
\end{align}
2) {\bf chordal extension}: Let
\begin{equation}\label{sec4-graph}
    H_{\hat{d},j}^{(k)}:=\overline{K_{\hat{d},j}^{(k)}}.
\end{equation}

Let $r_j=|\W_{\hat{d}-d_j}|$. As for eigenvalue optimization, we then consider the following series of sparse moment relaxations for ($\textrm{TQ}_{\hat{d}}$) indexed by $k\ge1$:
\begin{equation}\label{cpop-strace1}
(\textrm{TQ}_{\hat{d},k}^{\textrm{ts}}):
\begin{array}{rll}
\mu_{\hat{d},k}^{\textrm{ts}}(f, S):=&\inf &L_{\y}(f)\\
&\textrm{s.t.}&B_{H_{\hat{d},0}^{(k)}}\circ M_{\hat{d}}(\y)\in\Pi_{H_{\hat{d},0}^{(k)}}(\S_+^{r_0}),\\
&&B_{H_{\hat{d},j}^{(k)}}\circ M_{\hat{d}-d_j}(g_j\y)\in\Pi_{H_{\hat{d},j}^{(k)}}(\S_+^{r_j}), j\in[m],\\
&&[B_{H_{\hat{d},0}^{(k)}}\circ M_{\hat{d}}(\y)]_{uv}=[B_{H_{\hat{d},0}^{(k)}}\circ M_{\hat{d}}(\y)]_{wz},\textrm{for all }u^{\star}v\stackrel{\textrm{cyc}}{\sim} w^{\star}z,\\
&&y_{1}=1.
\end{array}
\end{equation}
We call $k$ the {\em sparse order}. By construction, one has $H_{\hat{d},j}^{(k)}\subseteq H_{\hat{d},j}^{(k+1)}$ for all $j,k$. Therefore, for every $j$, the sequence of graphs
$(H_{\hat{d},j}^{(k)})_{k\ge1}$ stabilizes after a finite number of steps. We denote the stabilized graphs by $H_{\hat{d},j}^{(\circ)}$ for all $j$ and the optimum of the corresponding SDP relaxation by $\mu_{\hat{d},\circ}(f, S)$.

For each $k\ge1$, the dual of $(\textrm{TQ}_{\hat{d},k}^{\textrm{ts}})$ reads as:
\begin{equation}\label{cpop-strace2}
(\textrm{TQ}_{\hat{d},k}^{\textrm{ts}})^*:
\begin{array}{rll}
&\sup&\mu\\
&\textrm{s.t.}&\sum_{w\stackrel{\textrm{cyc}}{\sim} v}(\sum_{j=0}^m\langle Q_j,D_{w}^j\rangle+\mu\delta_{1w})=\sum_{w\stackrel{\textrm{cyc}}{\sim} v}a_{w},\\
&&\quad\quad\quad\quad\quad\quad\quad\quad\quad\quad\forall v\in[\bigcup_{j=0}^m\supp_{g_j}(H_{\hat{d},j}^{(k)}))\cup\W^2_{\hat{d}}],\\
&&Q_j\in\S_+^{r_j}\cap\S_{H_{\hat{d},j}^{(k)}},\quad j\in\{0\}\cup[m],
\end{array}
\end{equation}
where $D_{w}^j$ is defined as in Section~\ref{sec2-eo}.

As for eigenvalue optimization, we have
\begin{theorem}\label{sec4-thm3}
Assume that $\{f\}\cup S\in\Sym\,\R\langle\underline{X}\rangle$. Then the followings hold:
\begin{enumerate}[(i)]
   \item Assume that $S$ is feasible and contains a polynomial $g_1=R^2-\sum_{i=1}^nX_i^2$ for some $R>0$. Then for all $\hat{d},k$, there is no duality gap between $(\textrm{TQ}_{\hat{d},k}^{\textrm{ts}})$ and $(\textrm{TQ}_{\hat{d},k}^{\textrm{ts}})^*$.
    \item Fixing a relaxation order $\hat{d}\ge d$, the sequence $(\mu^{\textrm{ts}}_{\hat{d},k}(f, S))_{k\ge1}$ is monotone nondecreasing and $\mu^{\textrm{ts}}_{\hat{d},k}(f, S)\le\mu_{\hat{d}}(f, S)$ for all $k$ (with $\mu_{\hat{d}}(f, S)$ defined in \eqref{cpop-trace1}).
    \item Fixing a sparse order $k\ge1$, the sequence $(\mu^{\textrm{ts}}_{\hat{d},k}(f, S))_{\hat{d}\ge d}$ is monotone nondecreasing.
    \item If the maximal chordal extension is used in \eqref{sec4-graph}, then $(\mu^{\textrm{ts}}_{\hat{d},k}(f, S))_{k\ge 1}$ converges to $\mu_{\hat{d}}(f, S)$ in finitely many steps, i.e., $\mu_{\hat{d},\circ}(f, S)=\mu_{\hat{d}}(f, S)$.
\end{enumerate}
\end{theorem}

Following from Theorem \ref{sec4-thm3}, we have the following two-level hierarchy of lower bounds for the optimum $\tr_{\min}(f,S)^{\II_1}$:
\begin{equation}\label{cliquehierct}
\begin{matrix}
\mu^{\textrm{ts}}_{d,1}(f, S)&\le&\mu^{\textrm{ts}}_{d,2}(f, S)&\le&\cdots&\le&\mu_{d}(f, S)\\
\vge&&\vge&&&&\vge\\
\mu^{\textrm{ts}}_{d+1,1}(f, S)&\le&\mu^{\textrm{ts}}_{d+1,2}(f, S)&\le&\cdots&\le&\mu_{d+1}(f, S)\\
\vge&&\vge&&&&\vge\\
\vdots&&\vdots&&\vdots&&\vdots\\
\vge&&\vge&&&&\vge\\
\mu^{\textrm{ts}}_{\hat{d},1}(f, S)&\le&\mu^{\textrm{ts}}_{\hat{d},2}(f, S)&\le&\cdots&\le&\mu_{\hat{d}}(f, S)\\
\vge&&\vge&&&&\vge\\
\vdots&&\vdots&&\vdots&&\vdots\\
\end{matrix}
\end{equation}
The array of lower bounds \eqref{cliquehierct} (and its associated moment-SOHS relaxations \eqref{cpop-strace1}-\eqref{cpop-strace2}) is what we call the NCTSSOS hierarchy associated with $(\textrm{TQ}_0)$.

\begin{example}\label{sec4-ex2}
Consider $f=2-X^2+XY^2X-Y^2$ and $S=\{4-X^2-Y^2,XY+YX-2\}$. We draw the graph $H_{2,0}^{(0)}$ for $f$ and $S$ in Figure \ref{sec4-ex3}. Since $H_{2,0}^{(0)}$ is already a chordal graph, we don't need any chordal extension. Hence $H_{2,0}^{(1)}=H_{2,0}^{(0)}$ and $(H_{2,j}^{(k)})_{k\ge1}$ immediately stabilizes at $k=1$ for all $j$. We compute that $\mu^{\textrm{ts}}_{2,1}(f, S)=\mu_{2}(f, S)=-1$.

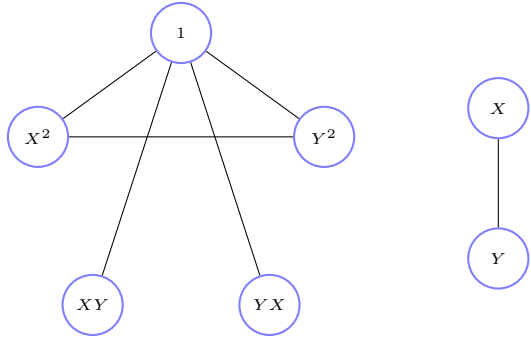
\begin{figure}[htbp]
\caption{The graph $H_{2,0}^{(0)}$ for Example \ref{sec4-ex2}}\label{sec4-ex3}
\begin{center}
{\tiny
\begin{tikzpicture}[every node/.style={circle, draw=blue!50, thick, minimum size=8mm}]
\node (n1) at (90:2) {$1$};
\node (n2) at (162:2) {$X^2$};
\node (n3) at (234:2) {$XY$};
\node (n4) at (306:2) {$YX$};
\node (n5) at (18:2) {$Y^2$};
\draw (n3)--(n1);
\draw (n4)--(n1);
\draw (n5)--(n1);
\draw (n1)--(n2);
\draw (n2)--(n5);
\node[xshift=120] (n6) at (90:1) {$X$};
\node[xshift=120] (n7) at (270:1) {$Y$};
\draw (n6)--(n7);
\end{tikzpicture}}
\end{center}
\end{figure} 
\end{example}

\subsection{Combining correlative sparsity with term sparsity}
We can also combine correlative sparsity with term sparsity for trace optimization. Let $\A$, $I_l,J_l,\A_l$, $\W_{\hat{d},l},\W_{\hat{d}-d_j,l}$ be defined as in Section~\ref{cts}. Fix a relaxation order $\hat{d}\ge d$. Let $H_{\hat{d},l}^{\textrm{tsp}}$ be the cyclic tsp graph with nodes $\W_{\hat{d},l}$ associated with $\A_l$ defined as in Section~\ref{trace-con}.
Assume that $H_{\hat{d},l,0}^{(0)}=H_{\hat{d},l}^{\textrm{tsp}}$ and $H_{\hat{d},l,j}^{(0)},j\in J_l, l\in[p]$ are empty graphs. Letting
\begin{equation}\label{trace-eq2}
    \DD_{\hat{d}}^{(k-1)}:=\bigcup_{l=1}^p\bigcup_{j\in \{0\}\cup J_l}\supp_{g_j}(H_{\hat{d},l,j}^{(k-1)})\cup\W_{\hat{d}}^2,\quad k\ge1,
\end{equation}
we recursively define a sequence of graphs $(H_{\hat{d},l,j}^{(k)}(V_{\hat{d},l,j},E_{\hat{d},l,j}^{(k)}))_{k\ge1}$ with $V_{\hat{d},l,j}=\W_{\hat{d}-d_j,l}$ for $j\in\{0\}\cup J_l,l\in[p]$ by
\begin{equation}\label{trace-eq3}
H_{\hat{d},l,j}^{(k)}:=\overline{K_{\hat{d},l,j}^{(k)}},
\end{equation}
where $K_{\hat{d},l,j}^{(k)}$ is the graph with $V(K_{\hat{d},l,j}^{(k)})=\W_{\hat{d}-d_j,l}$ and
\begin{equation}\label{cts-teq4}
E(K_{\hat{d},l,j}^{(k)})=\{\{u,v\}\mid(u,v)\in\W_{\hat{d}-d_j,l}\times \W_{\hat{d}-d_j,l}, [u^{\star}\supp(g_j)v]\cap[\DD_{\hat{d}}^{(k-1)}]\ne\emptyset\}.
\end{equation}

Let $r_{l,j}=|\W_{\hat{d}-d_j,l}|$ for all $l,j$. Then for each $k\ge1$, the moment relaxation based on combined correlative-term sparsity for $(\textrm{TQ}_0)$ is defined as:
\begin{equation}\label{trace-eq4}
(\textrm{TQ}^{\textrm{cs-ts}}_{\hat{d},k}):
\begin{array}{rll}
\mu^{\textrm{cs-ts}}_{\hat{d},k}(f, S):=
&\inf&L_{\y}(f)\\
&\textrm{s.t.}&B_{H_{\hat{d},l,0}^{(k)}}\circ M_{\hat{d}}(\y, I_l)\in\Pi_{H_{\hat{d},l,0}^{(k)}}(\S_+^{r_{l,0}}), l\in[p],\\
&&B_{H_{\hat{d},l,j}^{(k)}}\circ M_{\hat{d}-d_j}(g_j\y, I_l)\in\Pi_{H_{\hat{d},l,j}^{(k)}}(\S_+^{r_{l,j}}), j\in J_l,l\in[p],\\
&&[B_{H_{\hat{d},l,0}^{(k)}}\circ M_{\hat{d}}(\y,I_l)]_{uv}=[B_{H_{\hat{d},l,0}^{(k)}}\circ M_{\hat{d}}(\y,I_l)]_{wz},\\
&&\quad\quad\quad\quad\quad\quad\quad\quad\quad\quad\textrm{for all }u^{\star}v\stackrel{\textrm{cyc}}{\sim} w^{\star}z,l\in[p],\\
&&y_1=1.
\end{array}
\end{equation}

For each $k\ge1$, the dual of $(\textrm{TQ}^{\textrm{cs-ts}}_{\hat{d},k})$ reads as:
\begin{equation}\label{trace-eq5}
(\textrm{TQ}^{\textrm{cs-ts}}_{\hat{d},k})^*:
\begin{cases}
\sup\,&\mu\\
\textrm{s.t.}\, &\sum_{w\stackrel{\textrm{cyc}}{\sim} v}(\sum_{l=1}^p\sum_{j\in \{0\}\cup J_l}\langle Q_{l,j},D_{w}^{l,j}\rangle+\mu\delta_{1w})=\sum_{w\stackrel{\textrm{cyc}}{\sim} v}a_{w},\forall v\in[\DD_{\hat{d}}^{(k)}],\\
&Q_{l,j}\in\S_+^{r_{l,j}}\cap\S_{H_{\hat{d},l,j}^{(k)}}, j\in \{0\}\cup J_l, l\in[p],
\end{cases}
\end{equation}
where $D_{w}^{l,j}$ is defined as in Section~\ref{cts} and $\DD_{\hat{d}}^{(k)}$ is defined as in \eqref{trace-eq2}.

\begin{theorem}\label{trace-thm1}
Assume that $\{f\}\cup S\subseteq\Sym\,\R\langle\underline{X}\rangle$. Let $\mu^{\textrm{cs}}_{\hat{d}}(f, S)$ be the optimum of the $\hat{d}$-th order sparse moment relaxation based on correlative sparsity for $(\textrm{TQ}_{0})$. Then the followings hold:
\begin{enumerate}[(i)]
    \item Fixing a relaxation order $\hat{d}\ge d$, the sequence $(\mu^{\textrm{cs-ts}}_{\hat{d},k}(f, S))_{k\ge1}$ is monotone non-decreasing and $\mu^{\textrm{cs-ts}}_{\hat{d},k}(f, S)\le\mu^{\textrm{cs}}_{\hat{d}}(f, S)$ for all $k$.
    \item Fixing a sparse order $k\ge 1$, the sequence $(\mu^{\textrm{cs-ts}}_{\hat{d},k}(f, S))_{\hat{d}\ge d}$ is monotone non-decreasing.
    \item If the maximal chordal extension is used in \eqref{trace-eq3}, then $(\mu^{\textrm{cs-ts}}_{\hat{d},k}(f, S))_{k\ge 1}$ converges to $\mu^{\textrm{cs}}_{\hat{d}}(f, S)$ in finitely many steps.
\end{enumerate}
\end{theorem}

From Theorem \ref{trace-thm1}, we deduce the following two-level hierarchy of lower bounds for the optimum $\tr_{\min}(f,S)^{\II_1}$:
\begin{equation}\label{trace-mixhierc}
\begin{matrix}
\mu^{\textrm{cs-ts}}_{d,1}(f, S)&\le&\mu^{\textrm{cs-ts}}_{d,2}(f, S)&\le&\cdots&\le&\mu^{\textrm{cs}}_{d}(f, S)\\
\vge&&\vge&&&&\vge\\
\mu^{\textrm{cs-ts}}_{d+1,1}(f, S)&\le&\mu^{\textrm{cs-ts}}_{d+1,2}(f, S)&\le&\cdots&\le&\mu^{\textrm{cs}}_{d+1}(f, S)\\
\vge&&\vge&&&&\vge\\
\vdots&&\vdots&&\vdots&&\vdots\\
\vge&&\vge&&&&\vge\\
\mu^{\textrm{cs-ts}}_{\hat{d},1}(f, S)&\le&\mu^{\textrm{cs-ts}}_{\hat{d},2}(f, S)&\le&\cdots&\le&\mu^{\textrm{cs}}_{\hat{d}}(f, S)\\
\vge&&\vge&&&&\vge\\
\vdots&&\vdots&&\vdots&&\vdots\\
\end{matrix}
\end{equation}

\section{Numerical Experiments}\label{sec:benchs}
In this section, we present numerical results of the proposed NCTSSOS hierarchies for both unconstrained and constrained noncommutative polynomial optimization problems.
Our tool to implement these hierarchies, named {\tt NCTSSOS}, is written as a Julia package. 
{\tt NCTSSOS} utilizes the Julia packages LightGraphs \cite{graph} to handle graphs, ChordalGraph \cite{Wang20} to generate an approximately minimum chordal extension and JuMP \cite{jump} to model SDP. Finally, {\tt NCTSSOS} relies on MOSEK \cite{mosek} to solve SDP. {\tt NCTSSOS} is freely available at

\vspace{2pt}
\centerline{https://github.com/wangjie212/NCTSSOS.}
\vspace{2pt}

All numerical examples were computed on an Intel Core i5-8265U@1.60GHz CPU with 8GB RAM memory. The timing includes the time for pre-processing (to get the block structure in {\tt NCTSSOS}), the time for modeling SDP and the time for solving SDP. For comparison purpose, we also implement the dense moment-SOHS relaxation in {\tt NCTSSOS}. The notations that we use are listed in Table \ref{table1}.
\begin{table}[htbp]
\caption{The notations}\label{table1}
\begin{center}
\begin{tabular}{|c|c|}
\hline
$n$&the number of variables\\
\hline
$k$&the sparse order\\
\hline
mb&the maximal size of blocks\\
\hline
opt&the optimal value\\
\hline
time&running time in seconds\\ 
\hline
$0$&a number with absolute value less than $\num{1e-4}$\\
\hline
-&out of memory\\
\hline
\end{tabular}
\end{center}
\end{table}

\subsection{Eigenvalue optimization examples}\label{eoe}
We first focus on the unconstrained case and consider the eigenvalue minimization problem for the following functions.

$\bullet$ The nc version of the Broyden banded function
\begin{equation*}
    f_{\textrm{Bb}}(\x)=\sum_{i=1}^n(2X_i+5X_i^3+1-\sum_{j\in J_i}(X_j+X_j^2))^{\star}(2X_i+5X_i^3+1-\sum_{j\in J_i}(X_j+X_j^2)),
\end{equation*}
where $J_i=\{j\mid j\ne i, \max(1,i-5)\le j\le\min(n,i+1)\}$.

$\bullet$ The nc version of the chained singular function
\begin{align*}
    f_{\textrm{cs}}(\x)=&\sum_{i\in J}((X_{i}+10X_{i+1})^{\star}(X_{i}+10X_{i+1})+5(X_{i+2}-X_{i+3})^{\star}(X_{i+2}-X_{i+3})\\&+(X^2_{i+1}-4X_{i+1}X_{i+2}+4X_{i+2}^2)^{\star}(X^2_{i+1}-4X_{i+1}X_{i+2}+4X_{i+2}^2)\\&+10(X^2_{i}-20X_iX_{i+3}+100X_{i+3}^2)^{\star}(X^2_{i}-20X_iX_{i+3}+100X_{i+3}^2),
\end{align*}
where $J=\{1,3,5,\ldots,n-3\}$.

$\bullet$ The nc version of the generalized Rosenbrock function
\begin{equation*}
    f_{\textrm{gR}}(\x)=1+\sum_{i=1}^n(100(X_i-X_{i-1}^2)^{\star}(X_i-X_{i-1}^2)+(1-X_i)^{\star}(1-X_i)).
\end{equation*}

$\bullet$ The nc version of the chained Wood function
\begin{align*}
    f_{\textrm{cW}}(\x)=&1+\sum_{i\in J}(100(X_{i+1}-X_{i}^2)^{\star}(X_{i+1}-X_{i}^2)+(1-X_i)^{\star}(1-X_i)+90(X_{i+3}\\
    &-X_{i+2}^2)^{\star}(X_{i+3}-X_{i+2}^2)+(1-X_{i+2})^{\star}(1-X_{i+2})+10(X_{i+1}+X_{i+3}\\
    &-2)^{\star}(X_{i+1}+X_{i+3}-2)+0.1(X_{i+1}-X_{i+3})^{\star}(X_{i+1}-X_{i+3})),
\end{align*}
where $J=\{1,3,5,\ldots,n-3\}$ and $4|n$.

$\bullet$ The nc version of the Broyden tridiagonal function
\begin{align*}
    f_{\textrm{Bt}}(\x)=&(3X_1-2X_1^2-2X_2+1)^{\star}(3X_1-2X_1^2-2X_2+1)\\
    &+\sum_{i=2}^{n-1}(3X_i-2X_i^2-X_{i-1}-2X_{i+1}+1)^{\star}(3X_i-2X_i^2-X_{i-1}-2X_{i+1}+1)\\
    &+(3X_n-2X_n^2-X_{n-1}+1)^{\star}(3X_n-2X_n^2-X_{n-1}+1).
\end{align*}

To solve the unconstrained eigenvalue minimization problem of these functions, we always rely on the Newton chip method to compute a monomial basis, which turns out to be much smaller than the standard monomial basis. We compute the optimal value $\lambda^{\textrm{ts}}_{1}(f)$ of $(\textrm{EQ}^{\textrm{ts}}_{1})$ 
using approximately minimum chordal extensions and compare the resulting values with the optimal value $\lambda_{\min}(f)$ of $(\textrm{EP})$ corresponding to the dense approach. The results are reported in Table \ref{eigen_Bb}--\ref{eigen_Bt}. It is evident from these tables that our sparse approach is much more scalable than the dense approach. The dense approach can never be executed due to the memory limit when the problem has over $100$ variables while the sparse approach can easily handle problems with $4000$ variables. Meanwhile when the dense approach is executable, the optimal value provided by the sparse approach is quite close (or even equal in many cases) to the one provided by the dense approach.

\begin{table}[htbp]
\caption{The eigenvalue minimization for the nc Broyden banded function}\label{eigen_Bb}
\begin{center}
\begin{tabular}{|c|c|c|c|c|c|c|}
\hline
\multirow{2}*{$n$}&\multicolumn{3}{c|}{sparse}&\multicolumn{3}{c|}{dense}\\
\cline{2-7}
&mb&opt&time&mb&opt&time\\
\hline
$20$&$15$&$0$&$0.34$&$61$&$0$&$1.42$\\
\hline
$40$&$15$&$0$&$0.77$&$121$&$0$&$34.9$\\
\hline
$60$&$15$&$0$&$0.97$&$181$&$0$&$367$\\
\hline
$80$&$15$&$0$&$1.20$&-&-&-\\
\hline
$100$&$15$&$0$&$1.57$&-&-&-\\
\hline
$200$&$15$&$0$&$3.14$&-&-&-\\
\hline
$300$&$15$&$0$&$5.25$&-&-&-\\
\hline
$400$&$15$&$0$&$7.11$&-&-&-\\
\hline
$500$&$15$&$0$&$9.42$&-&-&-\\
\hline
$600$&$15$&$0$&$12.9$&-&-&-\\
\hline
$700$&$15$&$0$&$15.6$&-&-&-\\
\hline
$800$&$15$&$0$&$18.5$&-&-&-\\
\hline
$900$&$15$&$0$&$22.3$&-&-&-\\
\hline
$1000$&$15$&$0$&$26.2$&-&-&-\\
\hline
\end{tabular}
\end{center}
\end{table}

\begin{table}[htbp]
\caption{The eigenvalue minimization for the nc chained singular function}\label{eigen_cs}
\begin{center}
\begin{tabular}{|c|c|c|c|c|c|c|}
\hline
\multirow{2}*{$n$}&\multicolumn{3}{c|}{sparse}&\multicolumn{3}{c|}{dense}\\
\cline{2-7}
&mb&opt&time&mb&opt&time\\
\hline
$20$&$3$&$-0.0004$&$0.06$&$59$&$-0.0001$&$1.65$\\
\hline
$40$&$3$&$-0.0024$&$0.10$&$119$&$-0.0003$&$54.0$\\
\hline
$60$&$3$&$0$&$0.16$&$179$&$-0.0002$&$516$\\
\hline
$80$&$3$&$-0.0005$&$0.19$&-&-&-\\
\hline
$100$&$3$&$0$&$0.20$&-&-&-\\
\hline
$200$&$3$&$-0.0001$&$0.50$&-&-&-\\
\hline
$400$&$3$&$-0.0331$&$0.97$&-&-&-\\
\hline
$600$&$3$&$-0.0005$&$1.85$&-&-&-\\
\hline
$800$&$3$&$-0.0381$&$2.69$&-&-&-\\
\hline
$1000$&$3$&$-0.0074$&$4.10$&-&-&-\\
\hline
$2000$&$3$&$-0.0004$&$15.7$&-&-&-\\
\hline
$3000$&$3$&$-0.0065$&$32.4$&-&-&-\\
\hline
$4000$&$3$&$-0.0007$&$58.7$&-&-&-\\
\hline
\end{tabular}
\end{center}
\end{table}

\begin{table}[htbp]
\caption{The eigenvalue minimization for the nc generalized Rosenbrock function}\label{eigen_gR}
\begin{center}
\begin{tabular}{|c|c|c|c|c|c|c|}
\hline
\multirow{2}*{$n$}&\multicolumn{3}{c|}{sparse}&\multicolumn{3}{c|}{dense}\\
\cline{2-7}
&mb&opt&time&mb&opt&time\\
\hline
$20$&$3$&$1.0000$&$0.06$&$40$&$1.0000$&$0.33$\\
\hline
$40$&$3$&$1.0000$&$0.06$&$80$&$1.0000$&$4.59$\\
\hline
$60$&$3$&$1.0000$&$0.07$&$120$&$1.0000$&$31.9$\\
\hline
$80$&$3$&$1.0000$&$0.08$&$160$&$1.0000$&$151$\\
\hline
$100$&$3$&$1.0000$&$0.08$&$200$&$1.0000$&$557$\\
\hline
$200$&$3$&$0.9999$&$0.15$&-&-&-\\
\hline
$400$&$3$&$0.9999$&$0.45$&-&-&-\\
\hline
$600$&$3$&$0.9999$&$0.70$&-&-&-\\
\hline
$800$&$3$&$0.9999$&$1.03$&-&-&-\\
\hline
$1000$&$3$&$0.9998$&$1.38$&-&-&-\\
\hline
$2000$&$3$&$1.0000$&$4.76$&-&-&-\\
\hline
$3000$&$3$&$1.0000$&$10.7$&-&-&-\\
\hline
$4000$&$3$&$0.9999$&$18.9$&-&-&-\\
\hline
\end{tabular}
\end{center}
\end{table}

\begin{table}[htbp]
\caption{The eigenvalue minimization for the nc chained Wood function}\label{eigen_cW}
\begin{center}
\begin{tabular}{|c|c|c|c|c|c|c|}
\hline
\multirow{2}*{$n$}&\multicolumn{3}{c|}{sparse}&\multicolumn{3}{c|}{dense}\\
\cline{2-7}
&mb&opt&time&mb&opt&time\\
\hline
$20$&$3$&$1.0000$&$0.05$&$31$&$1.0000$&$0.16$\\
\hline
$40$&$3$&$0.9997$&$0.08$&$61$&$1.0000$&$1.14$\\
\hline
$60$&$3$&$0.9992$&$0.09$&$91$&$1.0000$&$7.06$\\
\hline
$80$&$3$&$1.0000$&$0.10$&$121$&$1.0000$&$30.9$\\
\hline
$100$&$3$&$1.0000$&$0.10$&$151$&$1.0000$&$100$\\
\hline
$200$&$3$&$0.9978$&$0.16$&-&-&-\\
\hline
$400$&$3$&$0.9930$&$0.43$&-&-&-\\
\hline
$600$&$3$&$0.9871$&$0.71$&-&-&-\\
\hline
$800$&$3$&$0.9846$&$1.04$&-&-&-\\
\hline
$1000$&$3$&$0.9919$&$1.41$&-&-&-\\
\hline
$2000$&$3$&$0.9605$&$4.95$&-&-&-\\
\hline
$3000$&$3$&$0.9889$&$9.93$&-&-&-\\
\hline
$4000$&$3$&$0.9652$&$18.6$&-&-&-\\
\hline
\end{tabular}
\end{center}
\end{table}

\begin{table}[htbp]
\caption{The eigenvalue minimization for the nc Broyden tridiagonal function}\label{eigen_Bt}
\begin{center}
\begin{tabular}{|c|c|c|c|c|c|c|}
\hline
\multirow{2}*{$n$}&\multicolumn{3}{c|}{sparse}&\multicolumn{3}{c|}{dense}\\
\cline{2-7}
&mb&opt&time&mb&opt&time\\
\hline
$20$&$5$&$0$&$0.07$&$41$&$0$&$0.25$\\
\hline
$40$&$5$&$0$&$0.08$&$81$&$0$&$3.28$\\
\hline
$60$&$5$&$0$&$0.09$&$121$&$0$&$21.6$\\
\hline
$80$&$5$&$0$&$0.13$&$161$&$0$&$117$\\
\hline
$100$&$5$&$0$&$0.15$&$201$&$0$&$335$\\
\hline
$200$&$5$&$0$&$0.29$&-&-&-\\
\hline
$400$&$5$&$0$&$0.66$&-&-&-\\
\hline
$600$&$5$&$0$&$0.97$&-&-&-\\
\hline
$800$&$5$&$0$&$1.56$&-&-&-\\
\hline
$1000$&$5$&$0$&$2.17$&-&-&-\\
\hline
$2000$&$5$&$0$&$7.58$&-&-&-\\
\hline
$3000$&$5$&$0$&$17.0$&-&-&-\\
\hline
$4000$&$5$&$0$&$29.5$&-&-&-\\
\hline
\end{tabular}
\end{center}
\end{table}

Now let us consider the constrained case. Let $\D$ be the semialgebraic set defined by $\{1-X_1^2,\ldots,1-X_n^2,X_1-1/3,\ldots,X_n-1/3\}$ and the optimization problem is minimizing the eigenvalue of the nc Broyden banded function over $\D$.
We compute the optimal value $\lambda^{\textrm{cs-ts}}_{\hat{d}, 1}(f,S)$ of $(\textrm{EQ}^{\textrm{cs-ts}}_{\hat{d},1})$ using approximately minimum chordal extensions with $\hat{d}=3$ (the minimum relaxation order). The results are reported in Table \ref{ceigen_Bb}. 
To show the benefits of our method by contrast with the usual sparse approach based on correlative sparsity, we also display the results for the latter approach (i.e., $(\textrm{EQ}^{\textrm{cs}}_{\hat{d}})$) and the results for the dense approach in the table.
Again one can see from the table that our sparse approach is more scalable than the approach that exploits only correlative sparsity as well as the dense approach. Actually, the last two can never be executed due to the memory limit even when the problem has only $6$ variables.
\begin{table}[htbp]
\caption{The eigenvalue minimization for the nc Broyden banded function over $\D$}\label{ceigen_Bb}
\begin{center}
\begin{tabular}{|c|c|c|c|c|c|c|c|c|c|}
\hline
\multirow{2}*{$n$}&\multicolumn{3}{c|}{CS+TS}&\multicolumn{3}{c|}{CS}&\multicolumn{3}{c|}{dense}\\
\cline{2-10}
&mb&opt&time&mb&opt&time&mb&opt&time\\
\hline
$5$&$11$&$3.113$&$0.50$&$156$&$3.113$&$70.7$&$156$&$3.113$&$69.8$\\
\hline
$10$&$15$&$3.011$&$2.78$&$400$&-&-&-&-&-\\
\hline
$20$&$15$&$9.658$&$11.4$&$400$&-&-&-&-&-\\
\hline
$30$&$15$&$16.30$&$22.3$&$400$&-&-&-&-&-\\
\hline
$40$&$15$&$22.94$&$38.1$&$400$&-&-&-&-&-\\
\hline
$50$&$15$&$29.57$&$57.7$&$400$&-&-&-&-&-\\
\hline
$60$&$15$&$36.21$&$80.5$&$400$&-&-&-&-&-\\
\hline
$70$&$15$&$42.85$&$105$&$400$&-&-&-&-&-\\
\hline
$80$&$15$&$49.49$&$138$&$400$&-&-&-&-&-\\
\hline
$90$&$15$&$56.13$&$151$&$400$&-&-&-&-&-\\
\hline
$100$&$15$&$62.77$&$180$&$400$&-&-&-&-&-\\
\hline
\end{tabular}\\
{\small In this table, ``CS+TS" indicates the results for the approach that exploits combined term-correlative sparsity; ``CS" indicates the results for the approach that exploits only correlative sparsity.}
\end{center}
\end{table}

\subsection*{Randomly generated examples.}
We construct randomly generated examples whose csp graph consists of $l$ maximal cliques of size $15$ as follows: let $f=\sum_{l=1}^p(h_l+h_l^{\star})/2$ where $h_l\in\R\langle X_{10l-9},\ldots,X_{10l+5}\rangle$ is a
random quartic polynomials with $15$ terms and coefficients taken from $[-1,1]$, and let $S=\{g_l\}_{l=1}^p$ where $g_l=1-X_{10l-9}^2-\ldots-X_{10l+5}^2$. We consider the eigenvalue minimization problem for $f$ over the multi-ball $\BB$ defined by $S$. Let $l=50,100,\ldots,400$ so that we obtain $8$ such instances\footnote{The polynomials can be downloaded at https://wangjie212.github.io/jiewang/code.html.}. 
We compute the NCTSSOS hierarchy $(\lambda^{\textrm{cs-ts}}_{\hat{d},k}(f, S))_{k\ge1}$ with $\hat{d}=2$ and report the results of the first three steps (where we use the maximal chordal extension for the first step and use approximate minimum chordal extensions for the second and third steps, respectively) in Table \ref{ceigen_rge}. As one may expect, neither the dense approach nor the approach that exploits only correlative sparsity can handle problems with so large sizes. On the other hand, our sparse approach is scalable up to $4005$ variables.

\begin{table}[htbp]
\caption{The eigenvalue minimization for randomly generated examples over multi-balls}\label{ceigen_rge}
\begin{center}
\begin{tabular}{|c|c|c|c|c|c|c|c|c|c|c|}
\hline
\multirow{2}*{$n$}&\multicolumn{4}{c|}{CS+TS}&\multicolumn{3}{c|}{CS}&\multicolumn{3}{c|}{dense}\\
\cline{2-11}
&$k$&mb&opt&time&mb&opt&time&mb&opt&time\\
\hline
\multirow{3}*{$505$}&$1$&$21$&$-15.91$&$3.26$&\multirow{3}*{$241$}&\multirow{3}*{-}&\multirow{3}*{-}&\multirow{3}*{-}&\multirow{3}*{-}&\multirow{3}*{-}\\
&$2$&$21$&$-15.42$&$7.49$&&&&&&\\
&$3$&$21$&$-15.31$&$10.6$&&&&&&\\
\hline
\multirow{3}*{$1005$}&$1$&$25$&$-32.58$&$9.71$&\multirow{3}*{$241$}&\multirow{3}*{-}&\multirow{3}*{-}&\multirow{3}*{-}&\multirow{3}*{-}&\multirow{3}*{-}\\
&$2$&$25$&$-31.91$&$24.5$&&&&&&\\
&$3$&$25$&$-31.71$&$40.9$&&&&&&\\
\hline
\multirow{3}*{$1505$}&$1$&$26$&$-48.57$&$18.9$&\multirow{3}*{$241$}&\multirow{3}*{-}&\multirow{3}*{-}&\multirow{3}*{-}&\multirow{3}*{-}&\multirow{3}*{-}\\
&$2$&$26$&$-47.00$&$47.0$&&&&&&\\
&$3$&$26$&$-46.71$&$90.0$&&&&&&\\
\hline
\multirow{3}*{$2005$}&$1$&$25$&$-63.58$&$33.7$&\multirow{3}*{$241$}&\multirow{3}*{-}&\multirow{3}*{-}&\multirow{3}*{-}&\multirow{3}*{-}&\multirow{3}*{-}\\
&$2$&$25$&$-62.05$&$85.8$&&&&&&\\
&$3$&$25$&$-61.76$&$149$&&&&&&\\
\hline
\multirow{3}*{$2505$}&$1$&$23$&$-81.07$&$52.9$&\multirow{3}*{$241$}&\multirow{3}*{-}&\multirow{3}*{-}&\multirow{3}*{-}&\multirow{3}*{-}&\multirow{3}*{-}\\
&$2$&$23$&$-78.75$&$134$&&&&&&\\
&$3$&$23$&$-78.21$&$263$&&&&&&\\
\hline
\multirow{3}*{$3005$}&$1$&$23$&$-95.73$&$74.8$&\multirow{3}*{$241$}&\multirow{3}*{-}&\multirow{3}*{-}&\multirow{3}*{-}&\multirow{3}*{-}&\multirow{3}*{-}\\
&$2$&$23$&$-93.13$&$212$&&&&&&\\
&$3$&$23$&$-92.71$&$396$&&&&&&\\
\hline
\multirow{3}*{$3505$}&$1$&$24$&$-111.2$&$93.4$&\multirow{3}*{$241$}&\multirow{3}*{-}&\multirow{3}*{-}&\multirow{3}*{-}&\multirow{3}*{-}&\multirow{3}*{-}\\
&$2$&$24$&$-108.3$&$258$&&&&&&\\
&$3$&$24$&$-107.8$&$531$&&&&&&\\
\hline
\multirow{3}*{$4005$}&$1$&$25$&$-131.1$&$122$&\multirow{3}*{$241$}&\multirow{3}*{-}&\multirow{3}*{-}&\multirow{3}*{-}&\multirow{3}*{-}&\multirow{3}*{-}\\
&$2$&$25$&$-127.5$&$375$&&&&&&\\
&$3$&$25$&$-126.8$&$687$&&&&&&\\
\hline
\end{tabular}\\
\end{center}
{\small In this table, ``CS+TS" indicates the results for the approach that exploits combined term-correlative sparsity; ``CS" indicates the results for the approach that exploits only correlative sparsity.}
\end{table}

\subsection{Trace optimization examples}\label{toe}
Let us consider the unconstrained trace minimization for the nc Broyden banded function and the nc Broyden tridiagonal function. We compute the optimal value $\mu^{\textrm{cs-ts}}_{1}(f)$ of $(\textrm{EQ}^{\textrm{cs-ts}}_{1})$ using approximately minimum chordal extensions. The results are reported in Table \ref{trace_Bb}--\ref{trace_Bt}, respectively. As for eigenvalue minimization, the sparse approach is much more scalable than the dense approach, which actually can never be executed for these examples due to the memory limit.

\begin{table}[htbp]
\caption{The trace minimization for the nc Broyden banded function}\label{trace_Bb}
\begin{center}
\begin{tabular}{|c|c|c|c|c|c|c|}
\hline
\multirow{2}*{$n$}&\multicolumn{3}{c|}{sparse}&\multicolumn{3}{c|}{dense}\\
\cline{2-7}
&mb&opt&time&mb&opt&time\\
\hline
$10$&$29$&$0$&$1.91$&-&-&-\\
\hline
$20$&$29$&$0$&$9.72$&-&-&-\\
\hline
$30$&$29$&$0$&$18.2$&-&-&-\\
\hline
$40$&$29$&$0$&$34.3$&-&-&-\\
\hline
$50$&$29$&$0$&$46.2$&-&-&-\\
\hline
$60$&$29$&$0$&$65.4$&-&-&-\\
\hline
$70$&$29$&$0$&$79.5$&-&-&-\\
\hline
$80$&$29$&$0$&$99.1$&-&-&-\\
\hline
$90$&$29$&$0$&$118$&-&-&-\\
\hline
$100$&$29$&$0$&$150$&-&-&-\\
\hline
\end{tabular}
\end{center}
\end{table}


\begin{table}[htbp]
\caption{The trace minimization for the nc Broyden tridiagonal function}\label{trace_Bt}
\begin{center}
\begin{tabular}{|c|c|c|c|c|c|c|}
\hline
\multirow{2}*{$n$}&\multicolumn{3}{c|}{sparse}&\multicolumn{3}{c|}{dense}\\
\cline{2-7}
&mb&opt&time&mb&opt&time\\
\hline
$20$&$6$&$0$&$0.16$&-&-&-\\
\hline
$40$&$6$&$0$&$0.27$&-&-&-\\
\hline
$60$&$6$&$0$&$0.36$&-&-&-\\
\hline
$80$&$6$&$0$&$0.44$&-&-&-\\
\hline
$100$&$6$&$0$&$0.57$&-&-&-\\
\hline
$200$&$6$&$0$&$1.36$&-&-&-\\
\hline
$400$&$6$&$0$&$3.48$&-&-&-\\
\hline
$600$&$6$&$0$&$7.28$&-&-&-\\
\hline
$800$&$6$&$0$&$10.9$&-&-&-\\
\hline
$1000$&$6$&$0$&$15.4$&-&-&-\\
\hline
$2000$&$6$&$0$&$55.9$&-&-&-\\
\hline
$3000$&$6$&$0$&$122$&-&-&-\\
\hline
$4000$&$6$&$0$&$220$&-&-&-\\
\hline
\end{tabular}
\end{center}
\end{table}


We next consider the trace minimization for the nc Broyden banded function over the semialgebraic set $\D$ defined in Section~\ref{eoe}. 
We compute the optimal value $\mu^{\textrm{cs-ts}}_{\hat{d},1}(f,S)$ of $(\textrm{TQ}^{\textrm{cs-ts}}_{\hat{d},1})$ using approximately minimum chordal extensions and compare with the results for the approach that exploits only correlative sparsity and the results for the dense approach. The minimum relaxation order $\hat{d}=3$ is used. 
The results are reported in Table \ref{ctrace_Bb}, which again demonstrate the scalability of our sparse approach.

\begin{table}[htbp]
\caption{The eigenvalue minimization for the nc Broyden banded function over $\D$}\label{ctrace_Bb}
\begin{center}
\begin{tabular}{|c|c|c|c|c|c|c|c|c|c|}
\hline
\multirow{2}*{$n$}&\multicolumn{3}{c|}{CS+TS}&\multicolumn{3}{c|}{CS}&\multicolumn{3}{c|}{dense}\\
\cline{2-10}
&mb&opt&time&mb&opt&time&mb&opt&time\\
\hline
$5$&$19$&$3.113$&$0.66$&$156$&$3.113$&$7.24$&$156$&$3.113$&$7.10$\\
\hline
$10$&$29$&$3.011$&$5.88$&$400$&-&-&-&-&-\\
\hline
$20$&$29$&$9.833$&$32.9$&$400$&-&-&-&-&-\\
\hline
$30$&$29$&$16.47$&$49.5$&$400$&-&-&-&-&-\\
\hline
$40$&$29$&$23.11$&$73.4$&$400$&-&-&-&-&-\\
\hline
$50$&$29$&$29.75$&$111$&$400$&-&-&-&-&-\\
\hline
$60$&$29$&$36.39$&$151$&$400$&-&-&-&-&-\\
\hline
$70$&$29$&$43.03$&$198$&$400$&-&-&-&-&-\\
\hline
$80$&$29$&$49.67$&$238$&$400$&-&-&-&-&-\\
\hline
$90$&$29$&$56.31$&$298$&$400$&-&-&-&-&-\\
\hline
$100$&$29$&$62.95$&$338$&$400$&-&-&-&-&-\\
\hline
\end{tabular}\\
{\small In this table, ``CS+TS" indicates the results for the approach that exploits combined term-correlative sparsity; ``CS" indicates the results for the approach that exploits only correlative sparsity.}
\end{center}
\end{table}

Finally, we consider the trace minimization for the randomly generated  quartic nc polynomials in Section~\ref{eoe} over the multi-ball $\B$. 
We compute the NCTSSOS hierarchy $(\mu^{\textrm{ts}}_{\hat{d}, k}(f,S))_{k\ge1}$ with the relaxation order $\hat{d}=2$. 
We report the results of the first three steps (where we always use approximate minimum chordal extensions) in Table \ref{ctrace_rge}. 
As one could expect, neither the dense approach nor the approach that exploits only correlative sparsity can handle these problems. On the other hand, our sparse approach is easily scalable up to $4005$ variables.

\begin{table}[htbp]
\caption{The trace minimization for randomly generated examples over multi-balls}\label{ctrace_rge}
\begin{center}
\begin{tabular}{|c|c|c|c|c|c|c|c|c|c|c|}
\hline
\multirow{2}*{$n$}&\multicolumn{4}{c|}{CS+TS}&\multicolumn{3}{c|}{CS}&\multicolumn{3}{c|}{dense}\\
\cline{2-11}
&$k$&mb&opt&time&mb&opt&time&mb&opt&time\\
\hline
\multirow{3}*{$505$}&$1$&$16$&$-4.997$&$4.94$&\multirow{3}*{241}&\multirow{3}*{-}&\multirow{3}*{-}&\multirow{3}*{-}&\multirow{3}*{-}&\multirow{3}*{-}\\
&$2$&$17$&$-4.983$&$7.40$&&&&&&\\
&$3$&$17$&$-4.975$&$7.66$&&&&&&\\
\hline
\multirow{3}*{$1005$}&$1$&$16$&$-10.14$&$14.2$&\multirow{3}*{241}&\multirow{3}*{-}&\multirow{3}*{-}&\multirow{3}*{-}&\multirow{3}*{-}&\multirow{3}*{-}\\
&$2$&$17$&$-10.11$&$21.7$&&&&&&\\
&$3$&$17$&$-10.11$&$22.6$&&&&&&\\
\hline
\multirow{3}*{$1505$}&$1$&$16$&$-15.72$&$25.2$&\multirow{3}*{241}&\multirow{3}*{-}&\multirow{3}*{-}&\multirow{3}*{-}&\multirow{3}*{-}&\multirow{3}*{-}\\
&$2$&$17$&$-15.68$&$39.8$&&&&&&\\
&$3$&$17$&$-15.67$&$41.0$&&&&&&\\
\hline
\multirow{3}*{$2005$}&$1$&$16$&$-20.45$&$40.9$&\multirow{3}*{241}&\multirow{3}*{-}&\multirow{3}*{-}&\multirow{3}*{-}&\multirow{3}*{-}&\multirow{3}*{-}\\
&$2$&$17$&$-20.41$&$67.9$&&&&&&\\
&$3$&$17$&$-20.40$&$73.8$&&&&&&\\
\hline
\multirow{3}*{$2505$}&$1$&$16$&$-25.95$&$63.1$&\multirow{3}*{241}&\multirow{3}*{-}&\multirow{3}*{-}&\multirow{3}*{-}&\multirow{3}*{-}&\multirow{3}*{-}\\
&$2$&$17$&$-25.90$&$95.6$&&&&&&\\
&$3$&$18$&$-25.89$&$101$&&&&&&\\
\hline
\multirow{3}*{$3005$}&$1$&$16$&$-31.09$&$93.5$&\multirow{3}*{241}&\multirow{3}*{-}&\multirow{3}*{-}&\multirow{3}*{-}&\multirow{3}*{-}&\multirow{3}*{-}\\
&$2$&$17$&$-31.03$&$152$&&&&&&\\
&$3$&$18$&$-31.02$&$157$&&&&&&\\
\hline
\multirow{3}*{$3505$}&$1$&$16$&$-35.99$&$119$&\multirow{3}*{241}&\multirow{3}*{-}&\multirow{3}*{-}&\multirow{3}*{-}&\multirow{3}*{-}&\multirow{3}*{-}\\
&$2$&$17$&$-35.93$&$198$&&&&&&\\
&$3$&$18$&$-35.92$&$216$&&&&&&\\
\hline
\multirow{3}*{$4005$}&$1$&$16$&$-41.80$&$145$&\multirow{3}*{241}&\multirow{3}*{-}&\multirow{3}*{-}&\multirow{3}*{-}&\multirow{3}*{-}&\multirow{3}*{-}\\
&$2$&$17$&$-41.72$&$248$&&&&&&\\
&$3$&$18$&$-41.70$&$264$&&&&&&\\
\hline
\end{tabular}\\
{\small In this table, ``CS+TS" indicates the results for the approach that exploits combined term-correlative sparsity; ``CS" indicates the results for the approach that exploits only correlative sparsity.}
\end{center}
\end{table}

\section{Conclusions and Outlooks}\label{cons}
We have presented the sparsity (term sparsity and combined correlative-term sparsity) adapted moment-SOHS hierarchies for both eigenvalue optimization and trace optimization involving noncommutative polynomials. Numerical experiments demonstrate that these sparse hierarchies are very efficient and scale well with the problem size when sparsity is present. One question left for future investigation is to develop a Gelfand-Naimark-Segal's style construction for extracting a minimizer adapted to our sparse setting.

Recently a moment-SOHS hierarchy for optimization problems involving trace polynomials was proposed in \cite{klep2020}. It would be worth extending further our sparsity-exploiting framework to handle trace polynomials.

We also plan to use the sparsity adapted moment-SOHS hierarchies developed in this paper to tackle large-scale NCPOPs arising from quantum information and condensed matter physics.

~

\paragraph{\textbf{Acknowledgements}.} 
Both authors were supported by the Tremplin ERC Stg Grant ANR-18-ERC2-0004-01 (T-COPS project).
The second author was supported by the FMJH Program PGMO (EPICS project) and  EDF, Thales, Orange et Criteo.
This work has benefited from  the European Union's Horizon 2020 research and innovation programme under the Marie Sklodowska-Curie Actions, grant agreement 813211 (POEMA) as well as from the AI Interdisciplinary Institute ANITI funding, through the French ``Investing for the Future PIA3'' program under the Grant agreement n$^{\circ}$ANR-19-PI3A-0004.

\bibliographystyle{plain}

\end{document}